\documentclass[reqno,11pt]{amsart}
 
\usepackage{a4wide}
\usepackage[latin1]{inputenc}
\usepackage{amsmath,amsthm,amssymb,bbm,mathrsfs}
\usepackage{color,epsfig,epstopdf}
\usepackage{ucs}
\usepackage{units}
\usepackage{stmaryrd}
\usepackage{enumitem}
\usepackage{comment}

\definecolor{darkgreen}{rgb}{0.00,0.50,0.10}
\definecolor{lightgreen}{rgb}{0.20,0.70,0.30}

\newtheorem{theorem}                   {Theorem}

\newtheorem{lemma}           [theorem] {Lemma}   
\newtheorem{corollary}       [theorem] {Corollary}   
\newtheorem{proposition}     [theorem] {Proposition}  
\newtheorem{claim}            {Claim} 
\newtheorem{fact}			 [theorem]{Fact}
 
\newtheorem{definition}      [theorem] {Definition}

\newtheorem{remark}          [theorem] {Remark}

\theoremstyle{remark} 
\newtheorem{subclaim}            {Subclaim}

\newcommand{\oldqed}{}
\def\endofClaim{\hfill\scalebox{.6}{$\Box$}}

\newenvironment{claimproof}[1][Proof]{
  \renewcommand{\oldqed}{\qedsymbol}
  \renewcommand{\qedsymbol}{\endofClaim}
  \begin{proof}[#1]
}{
  \end{proof}
  \renewcommand{\qedsymbol}{\oldqed}
}

\let\subset\subseteq
\let\epsilon\varepsilon
\let\eps\varepsilon
\let\rho\varrho
\def\LpSpace{L}
\def\dist{\mathrm{dist}\:}
\def\SUPPORT{\mathrm{supp}\:}

\setitemize{leftmargin=\parindent}
\setenumerate{leftmargin=*}

\newcommand{\JUSTIFY}[1]{\fbox{\tiny{#1}}\quad}

\newcommand{\By}[2]{\overset{\mbox{\tiny{#1}}}{#2}}
\newcommand{\ByRef}[2]{   \By{\eqref{#1}}{#2} }

\newcommand{\geByRef}[1]{ \ByRef{#1}{\ge} }

\newcommand{\Prob}{\mathbf{P}} 
\newcommand{\bbR}{\mathbb{R}} 

\newcommand{\frakc}{\mathfrak{c}} 
\newcommand{\calE}{\mathcal{E}} 

\newcommand{\fraccoverpolytope}{\mathrm{FCOV}} 
\newcommand{\coverpolytope}{\mathrm{COV}}
\newcommand{\matchingpolytope}{\mathrm{MATCH}} 
\newcommand{\perfmatchingpolytope}{\mathrm{PerfMATCH}} 
\newcommand{\fracperfmatchingpolytope}{\mathrm{FPerfMATCH}} 

\newcommand{\fraccoverratio}{\mathsf{fcov}} 
\newcommand{\matchratio}{\mathsf{match}}

\begin{document}

\title{Matching polytons}
\author{Martin Dole\v zal}
\address{Institute of Mathematics of the Czech Academy of Sciences, \v Zitn\'a 25, Praha. The Institute of Mathematics is supported by RVO:67985840.} 
\email{dolezal@math.cas.cz}\email{honzahladky@gmail.com}
\author{Jan Hladk\'y}

\thanks{{\it Martin Dole\v zal} was supported by GA \v CR Grant 16-07378S. 
\\	
{\it Jan Hladk\'y:} The research leading to these results has received funding from the People Programme (Marie Curie Actions) of the European Union's Seventh Framework Programme (FP7/2007-2013) under REA grant agreement number 628974.\\
An extended abstract describing these results will appear in the proceedings of the EuroComb2017 conference,~\cite{DoHlHuPi:CombinatorialOptimization}.
}

\begin{abstract}
Hladk\'y, Hu, and Piguet [Tilings in graphons, preprint] introduced the notions of matching and fractional vertex covers in graphons. These are counterparts to the corresponding notions in finite graphs. 

Combinatorial optimization studies the structure of the matching polytope and the fractional vertex cover polytope of a graph. Here, in analogy, we initiate the study of the structure of the set of all matchings and of all fractional vertex covers in a graphon. We call these sets the matching polyton and the fractional vertex cover polyton.

We also study properties of matching polytons and fractional vertex cover polytons along convergent sequences of graphons.

As an auxiliary tool of independent interest, we prove that a graphon is $r$-partite if and only if it contains no graph of chromatic number $r+1$. This in turn gives a characterization of bipartite graphons as those having a symmetric spectrum.
\end{abstract}
\maketitle


\section{Introduction}
Theories of graph limits are arguably one of the most important directions in discrete mathematics in the last decade. They link graph theory to analytic parts of mathematics, and through this connection introduce new tools to graph theory. With regard to such applications, the most fruitful theory has been that of flag algebras, \cite{Razborov2007}. Here, we deal with the theory of graphons developed by Borgs, Chayes, Lov\'asz, Szegedy, S\'os, and Vesztergombi, \cite{Borgs2008c,Lovasz2006}. This theory, too, has found numerous applications in extremal graph theory (e.g.,~\cite{Lov:Sidorenko}), theory of random graphs \cite{ChatVar:LargeDev}, and in our understanding of properties of Szemer\'edi regularity partitions, see e.g.~\cite{Lovasz2010}.
Much of the theory is built on counterparts of concepts well-known from the world of finite graphs (such as subgraph densities or cuts). Ideally, such counterparts are continuous with respect to the cut-distance, and are equal to the original concept when the graphon in question is a representation of a finite graph.

Hladk\'y, Hu, and Piguet in~\cite{HlHuPi:TilingsInGraphons} translated the concept of vertex-disjoint copies of a fixed finite graph $F$ in a (large) host graph to graphons. Following preceding literature on this topic, they use the name \emph{$F$-tiling} (in a graph or in a graphon). This allows them to introduce the \emph{$F$-tiling ratio} of a graphon. They also translate the closely related concept of (fractional) $F$-covers in finite graphs to graphons which is a dual concept (in the sense of linear programming) to $F$-tilings. The case when $F$ is an edge, $F=K_2$, is the most important. Then $F$-tilings are exactly matchings, the $F$-tiling ratio is just the matching ratio\footnote{the \emph{matching ratio} is just the matching number divided by the number of vertices}, and (fractional) $F$-covers are exactly (fractional) vertex covers. 

In this paper we deal exclusively with the case $F=K_2$ and from now on we specialize our description to this case. We discuss the possible generalizations to other $F$-tilings of some of the results presented in this paper in  Section~\ref{sub:GenerelazingToFTilings}.

Hladk\'y, Hu, and Piguet
mostly study the numerical quantities provided by the theory they develop, that is, the matching ratio and the fractional vertex cover ratio. Given a graphon $W$, we denote these two quantities (which we define in Section~\ref{sec:matchings}) by $\matchratio(W)$ and $\fraccoverratio(W)$. One of their main results is a transference result about the matching ratio between finite graphs and graphons, as follows.\footnote{See Section~\ref{sec:notation} for the definition of cut-distance convergence.}
\begin{theorem}[Theorem 3.4 in~\cite{HlHuPi:TilingsInGraphons}]\label{thm:tilingtransference}
	Suppose that $(G_n)_n$ is a sequence of graphs of growing orders that converge to a graphon $W$ in the cut-distance. Then for every $\epsilon>0$ there exists $n_0$ so that for each $n>n_0$ the graph $G_n$ contains a matching of size at least $(\frac{\matchratio(W)}{2}-\epsilon)\cdot v(G_n)$.
\end{theorem}	
Another main result from~\cite{HlHuPi:TilingsInGraphons} is the counterpart of the prominent linear programming duality between the fractional matching number of a graph and its fractional vertex cover number. Since --- as Hladk\'y, Hu, and Piguet argue --- in the graphon world there is no distinction between matchings and fractional matchings, their LP duality has the form $\matchratio(W)=\fraccoverratio(W)$. In~\cite{HlHuPi:TilingsInGraphons} and \cite{HlHuPi:Komlos} they give applications of this LP duality in extremal graph theory.\footnote{These applications become particularly interesting when general $F$-tilings are considered.}

In this paper, on the other hand, we study the sets of all fractional matchings, and of all fractional vertex covers. In the case of a finite graph $G$, these sets are known as the \emph{fractional matching polytope} and the \emph{fractional vertex cover polytope}. We shall denote them by $\matchingpolytope(G)$ and $\fraccoverpolytope(G)$. Study of $\matchingpolytope(G)$ and $\fraccoverpolytope(G)$ (and study of related polytopes such as the (integral) matching polytope and the perfect matching polytope) is central in polyhedral combinatorics and in combinatorial optimization. From the numerous results on the geometry of these polytopes, let us mention integrality of the fractional matching polytope and fractional vertex cover polytope of a bipartite graph, or the Edmonds' perfect matching polytope theorem. Here, we initiate a parallel study in the context of graphons. While in the finite case we have $\matchingpolytope(G)\subset \mathbb{R}^{E(G)}$ and $\fraccoverpolytope(G)\subset\mathbb{R}^{V(G)}$, given a graphon $W:\Omega^2\rightarrow [0,1]$, for the corresponding objects $\matchingpolytope(W)$ and $\fraccoverpolytope(W)$ it turns out that we have $\matchingpolytope(W)\subseteq\LpSpace^1(\Omega^2)$ and $\fraccoverpolytope(W)\subseteq\LpSpace^\infty(\Omega)$. So, while $\matchingpolytope(G)$ and $\fraccoverpolytope(G)$ are studied using tools from linear algebra, in order to study $\matchingpolytope(W)$ and $\fraccoverpolytope(W)$ we need to use the language of functional analysis. We employ the \emph{-on} word ending used among others for graph\emph{on}s and for permut\emph{on}s and call the limit counterparts to polytopes (such as $\matchingpolytope(W)$ and $\fraccoverpolytope(W)$) polyt\emph{on}s. 

\subsection{Overview of the paper} In Section~\ref{sec:notation} we recall the necessary background concerning graphons and the theory of matchings/tilings in graphons developed in~\cite{HlHuPi:TilingsInGraphons}. As an auxiliary tool for our main results, we prove in Section~\ref{ssec:characterizationofrpartite} that a graphon is $r$-partite if and only if it has zero density of every graph of chromatic number $r+1$ (see Proposition~\ref{prop:morecolors}). This in turn gives a characterization of bipartite graphons as those having a symmetric spectrum (see Theorem~8), which is a graphon counterpart to a well-known property for finite graphs.

In Section~\ref{sec:extremepoints} we treat (half)-integrality of the extreme points of the fractional vertex cover polyton of a graphon. The main results of this section are summarized in Theorems~\ref{thm:CovIntegral}, \ref{thm:CovHalfIntegral} and \ref{thm:CovIntegralImpliesBipartite}. As an application, we deduce a graphon version of the Erd\H{o}s--Gallai theorem on matchings in dense graphs (see Theorem~\ref{thm:EGgraphon}). In Section~\ref{sec:matchingpolytonconvergent} we show that if a sequence of graphons $(W_n)_n$ converges to a graphon $W$ in the cut-distance then ``$\matchingpolytope(W_n)$  asymptotically contains $\matchingpolytope(W)$'' (see Corollary~\ref{cor:assymptcont}, or Theorem~\ref{thm:convergencematchings} for a slightly stronger statement). This result is dual in the sense of linear programming to results from~\cite{HlHuPi:TilingsInGraphons} on the relation between $\fraccoverpolytope(W_n)$ and $\fraccoverpolytope(W)$, which we recall in Section~\ref{sub:coverpolytonconvergent}.

Section~\ref{sec:concluding} contains some concluding remarks.

\section{Notation and preliminaries}\label{sec:notation}
\subsection{Functional analysis}
\subsubsection{Weak* convergence and the Banach--Alaoglu theorem}\label{ssec:WeakConv}
We now recall the concept of weak* convergence and the Banach--Alaoglu theorem (which we present in a form that is tailored for our purposes). Suppose that $(X,\lambda)$ is a Borel probability space. Recall that a sequence of measurable functions $f_1,f_2,\ldots:X\rightarrow [0,1]$ \emph{converges weak*} to a measurable function $f:X\rightarrow [0,1]$ if for every measurable set $A\subset X$ we have $\lim_n \int_A f_n =\int_A f$. The Banach--Alaoglu theorem then asserts that the space of measurable functions from $X$ to $[0,1]$ is sequentially compact.

Since the product of finitely many sequentially compact spaces is sequentially compact, the Banach--Alaoglu theorem generalizes as follows. Suppose that $k\in\mathbb{N}$ is fixed, and suppose that we are given sequences $f^1_1,f^1_2,f^1_3,\ldots$, $f^2_1,f^2_2,f^2_3,\ldots$, \ldots, $f^k_1,f^k_2,f^k_3,\ldots$ of measurable functions from $X$ to $[0,1]$. Then there exist a sequence of indices $n_1<n_2<\ldots$ and functions $f^1,f^2,\ldots,f^k:X\rightarrow[0,1]$ such that for each $i\in[k]$, the sequence $f^i_{n_1},f^i_{n_2},f^i_{n_3},\ldots$ converges weak* to $f^i$.

\subsubsection{The Krein--Milman Theorem}\label{sssec:KreinMilman}
In this section, we briefly recall the Krein--Milman theorem. This theorem will allow to talk about ``vertices'' in our counterparts to polytopes.

Suppose that $\mathcal L$ is a vector space, and suppose that $X\subset \mathcal L$ is a convex set. Recall that a point $x\in X$ is called an \emph{extreme point of $X$} if the only pair $x',x''\in X$ for which $x=\frac12(x'+x'')$ is the pair $x'=x$, $x''=x$.  We shall write $\calE(X)$ to denote the set of all extreme points of $X$.

When $\mathcal L$ is finite-dimensional and $X$ is a polytope in $\mathcal L$ then the extreme points of $X$ are exactly its vertices. The importance of the notion of extreme points comes from the Krein--Milman theorem which states that in a locally convex topological vector space, each compact convex set equals the closed convex hull of its extreme points. 

The notion of extreme points is not the only generalization of vertices of a polytope. Another basic notion from convex analysis is that of exposed points. (Note that its stronger variant, the notion of strongly exposed points, can be used for a characterization of the Radon--Nikodym property of Banach spaces which is an extensively studied topic.) A point $x$ in a convex set $X$ is \emph{exposed} if there exists a continuous linear functional for which $x$ attains its strict maximum on $X$.

\begin{figure}
	\includegraphics[scale=1]{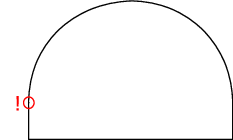} 
	\caption{By joining a half-circle and a rectangle in $\mathbb R^2$ we get an extreme point that is not exposed.}
	\label{fig:nonexposed}
\end{figure}
It is easy to see that every exposed point is extreme. The converse does not hold; a well-known counterexample in $\mathbb R^2$ is shown in Figure~\ref{fig:nonexposed}.

\subsection{Measure theory}
Suppose that $\lambda$ is a probability measure on a space $\Lambda$. Recall that for a measurable function $f:\Lambda\rightarrow \mathbb R$, its \emph{essential supremum} is defined as $$\mathrm{esssup}(f)=\inf\{a:\lambda(\{x:f(x)>a\})=0\}\;.$$
That is, the essential supremum is like the ordinary supremum, except that high values are not considered as long as they occur on a null-set. The \emph{essential infimum} $\mathrm{essinf}(f)$ is defined analogously.

Recall also that the product measure $\lambda^{\otimes 2}$ is defined by
$$\lambda^{\otimes 2}(R)=\inf\left\{\sum_{i=1}^\infty\lambda(A_i)\lambda(B_i)\colon A_i,B_i\subset\Lambda\text{ are measurable such that }R\subset\bigcup_{i=1}^\infty A_i\times B_i \right\}$$
for every $R\subset\Lambda^2$ from the product $\sigma$-algebra.

Below, we state an auxiliary result, which says that for every set $D\subset \Lambda^2$ we can find a rectangle (i.e., a set which is a product of two sets) which is almost entirely contained in $D$.
\begin{lemma}\label{lem:product_measure}
Let $(\Lambda,\lambda)$ be a probability space and let $A,B\subseteq\Lambda$ be given. Let $D\subseteq A\times B$ be a set of positive $\lambda^{\otimes 2}$ measure.
Then for every $\varepsilon>0$ there is a measurable rectangle $R\subseteq A\times B$ such that $\lambda^{\otimes 2}(R\setminus D)<\varepsilon\lambda^{\otimes 2}(R)$.
\end{lemma}

\begin{proof}
Let us fix $\varepsilon>0$. By the definition of the product measure, we can find measurable rectangles $R_1,R_2,\ldots\subseteq A\times B$ such that $D\subseteq\bigcup\limits_{i=1}^{\infty}R_i$ and
\begin{equation}\nonumber
\lambda^{\otimes 2}\left(\bigcup\limits_{i=1}^{\infty}R_i\setminus D\right)<\varepsilon\lambda^{\otimes 2}(D)\le
\varepsilon\lambda^{\otimes 2}\left(\bigcup\limits_{i=1}^{\infty}R_i\right)\;.
\end{equation}
Then there is a natural number $m$ such that
\begin{equation}\label{eq:nerovnost_pro_R_i}
\lambda^{\otimes 2}\left(\bigcup\limits_{i=1}^mR_i\setminus D\right)<\varepsilon\lambda^{\otimes 2}\left(\bigcup\limits_{i=1}^mR_i\right)\;.
\end{equation}
Now the finite union $\bigcup\limits_{i=1}^mR_i$ can obviously be decomposed into finitely many \emph{pairwise disjoint} measurable rectangles $S_1,\ldots,S_l$.
Then the inequality (\ref{eq:nerovnost_pro_R_i}) can be rewritten as
\begin{equation}\nonumber
\sum\limits_{i=1}^l\lambda^{\otimes 2}(S_i\setminus D)<\varepsilon\sum\limits_{i=1}^l\lambda^{\otimes 2}(S_i)\;.
\end{equation}
Thus there is some $i\in\{1,\ldots,l\}$ such that $\lambda^{\otimes 2}(S_i\setminus D)<\varepsilon\lambda^{\otimes 2}(S_i)$. The corresponding $S_i$ is the wanted measurable rectangle $R$.
\end{proof}

\subsection{Graphon basics}
Our notation follows~\cite{Lovasz2012}. Throughout the paper we
shall assume that $\Omega$ is an atomless Borel probability space
equipped with a measure $\nu$ (defined on an implicit $\sigma$-algebra). We denote by $\nu^{\otimes k}$ the product measure on $\Omega^k$.

A \emph{graphon} is a symmetric Lebesgue measurable function $W:\Omega^2\rightarrow[0,1]$. We refer to~\cite[Part 3]{Lovasz2012} for a detailed treatment of the concept. On an intuitive level, a graphon can be viewed as an adjacency matrix (scaled down to fit into $\Omega^2$) of a large graph, where the values of the graphon represent frequencies of 1's in the adjacency matrix. 

Suppose that $F$ is a graph on the vertex set $[k]$. Then the \emph{density} of $F$ in a graphon $W$ is defined as
$$t(F,W)=\int_{x_1}\int_{x_2}\cdots\int_{x_k}\;\prod_{\substack{ij\in E(F)\\ i<j}}W(x_i,x_j)\;.$$

Recall that the \emph{cut-norm} $\|\cdot\|_\square$ and the \emph{cut-distance} $\dist_\square(\cdot,\cdot)$ are defined by
\begin{align*}
\|U\|_\square&=\sup_{S,T\subset \Omega}\left|\int_{S\times T} U\right|\;,\;U\in\LpSpace^1(\Omega^2)\;,
\end{align*}
where the infimum ranges over all measurable sets $S,T\subset \Omega$, and
\begin{align*}
\dist_\square(U,W)&=\inf_\phi \|U-W^\phi\|_\square\;,\;U,W\in\LpSpace^1(\Omega^2)\;,
\end{align*}
where the infimum ranges over all measure-preserving bijections on $\Omega$, and $W^\phi$ is defined by $W^\phi(x,y)=W(\phi(x),\phi(y))$.

Given a finite graph $G$ we can construct its \emph{graphon representation} $W_G:\Omega^2\rightarrow[0,1]$ as follows. We partition $\Omega$ arbitrarily into sets $\{\Omega_v\}_{v\in V(G)}$ of measure $1/v(G)$ each. We then define $W_G$ to be $1$ on $\Omega_u\times \Omega_v$ if $uv\in E(G)$ and~$0$ otherwise. The actual graphon representation $W_G$ depends on the choice of the partition $\{\Omega_v\}_{v\in V(G)}$, but the cut-distance $\dist_\square(W_G,\cdot)$ does not.

\subsubsection{Independent sets and partite graphons}
A set $A\subset \Omega$ is an \emph{independent set} in a graphon $W:\Omega^2\rightarrow [0,1]$ if $W$ is zero almost everywhere on $A\times A$. In~\cite[Lemma~2.3]{HlHuPi:Komlos} it was proven that the support of any weak$^*$ accumulation point of a sequence of (indicator functions of) independent sets is an independent set. Here, we extend the statement by allowing the sets in the sequence to be only ``asymptotically independent''. The proof is however similar to that in~\cite{HlHuPi:Komlos}.
\begin{lemma}
	\label{lem:limitofalmostindependent}Let $W:\Omega^{2}\rightarrow[0,1]$
	be a graphon. Suppose that $\left(A_{n}\right)_{n=1}^{\infty}$ is
	a sequence of sets in $\Omega$ with the property that $$\int_{(x,y)\in A_n\times A_n}W(x,y) \overset{n\rightarrow \infty}{\longrightarrow} 0\;.$$ Suppose that the indicator
	functions of the sets $A_{n}$ converge weak$^{*}$ to a function
	$f:\Omega\rightarrow[0,1]$. Then $\SUPPORT f$ is an independent
	set in $W$.
\end{lemma}
\begin{proof}
	It is enough to prove that for every $\delta>0$, the set $P=\{x\in\Omega:f(x)>\delta\}$
	is independent. Suppose that the statement is false. Then there exist
	sets $X,Y\subset P$ of positive measure and a number $\alpha>0$
	such that
	\begin{equation}
	\nu\left(X\times Y\cap\left\{ (x,y)\in\Omega^{2}:W(x,y)\le\alpha\right\} \right)<\frac{\delta^{2}}{5}\nu(X)\nu(Y)\;.\label{eq:vlak}
	\end{equation}
	Recall that $\int_{X}f\ge\delta\nu(X)$ and $\int_{Y}f\ge\delta\nu(Y)$.
	By weak$^{*}$ convergence, for $n$ sufficiently large, $\nu(X\cap A_{n})\ge\frac{1}{2}\delta\nu(X)$
	and $\nu(Y\cap A_{n})\ge\frac{1}{2}\delta\nu(Y)$. In particular,
	(\ref{eq:vlak}) gives that
	\begin{equation}
	\int_{X\cap A_{n}}\int_{Y\cap A_{n}}W(x,y)\ge\left(\frac{1}{4}\delta^{2}-\frac{1}{5}\delta^{2}\right)\alpha\nu(X)\nu(Y)=\frac{1}{20}\delta^{2}\alpha\nu(X)\nu(Y)\;,\label{eq:c1}
	\end{equation}
	for each large enough $n$. On the other hand,
	\begin{equation}
	\int_{X\cap A_{n}}\int_{Y\cap A_{n}}W(x,y)\le\int_{A_{n}}\int_{A_{n}}W(x,y)\overset{n\rightarrow \infty}{\longrightarrow}0\;,\label{eq:c2}
	\end{equation}
	by the assumptions of the lemma. Clearly,~(\ref{eq:c1})
	contradicts (\ref{eq:c2}).
\end{proof}

Given a cardinal $r$, let us recall that a graphon $W:\Omega^2\rightarrow [0,1]$ is \emph{$r$-partite} if there exists a partition $\Omega=\bigsqcup_{i\in I}\Omega_i$, where for the index set $I$ we have $|I|=r$, into measurable sets such that for each $i\in I$, we have that $\Omega_i$ is an independent set in $W$. While most of the time we are interested in $r$ finite, note that this definition allows to define even \emph{countably-partite} graphons.

\subsubsection{Inhomogeneous random graphs} Let us recall that given a graphon $W:\Omega^2\rightarrow [0,1]$ and a number $n\in \mathbb{N}$, the \emph{$W$-random graph} is defined on the vertex set $\{1,\ldots,n\}$ as follows. First, draw independent random points $x_1,\ldots,x_n$ from $\Omega$ according to the measure $\nu$, and then for each $ij\in{n\choose 2}$, put an edge into the graph with probability $W(x_i,x_j)$ (and independently). For more details, see~\cite[Section 10.1]{Lovasz2012}.

\subsection{Characterization of partite graphons using forbidden subgraphs}\label{ssec:characterizationofrpartite}
In this section, we prove that for a given $r\in\mathbb N$, a graphon is $r$-partite if and only if it has zero density of each finite graph of chromatic number at least $r+1$. This is first proven for $r=2$ in Lemma~\ref{lemma:bipartite} (in which case the characterization of bipartiteness can be phrased in terms of odd cycles, just as in the world of finite graphs) and then for general $r$ in Proposition~\ref{prop:morecolors}. A strengthening of Lemma~\ref{lemma:bipartite} which we give in Lemma~\ref{lem:oddcycles} is needed for proving Theorem~\ref{thm:CovIntegralImpliesBipartite}. These results may be of independent interest.\footnote{In subsequent work~\cite{HlaRoch:IndepCliCol}, Hladk\'y and Rocha used this lemma to obtain properties of the ``independent set polyton''.} Further, in Section~\ref{sssection:spectrabipartite} we give a characterization of bipartite graphons in terms of their spectrum. 

\begin{lemma}\label{lemma:bipartite}
	Suppose that $W:\Omega^2\rightarrow [0,1]$ is a graphon. Then $W$ is bipartite if and only if for every odd integer $k\ge 3$ it holds $t(C_k,W)=0$.
\end{lemma}
\begin{proof}[Proof of Lemma~\ref{lemma:bipartite}]
	Suppose first that there is an odd integer $k\ge 3$ such that $t(C_k,W)>0$. Let $\Omega=\Omega_0\sqcup \Omega_1$ be an arbitrary decomposition of $\Omega$ into two disjoint measurable subsets. Then there exists $(i_j)_{j=1}^k\in\{0,1\}^k$ such that
	\begin{equation}\nonumber
	\int\limits_{x_1\in\Omega_{i_1}}\int\limits_{x_2\in\Omega_{i_2}}\cdots\int\limits_{x_k\in\Omega_{i_k}}W(x_1,x_2)W(x_2,x_3)\ldots W(x_k,x_1)>0\;.
	\end{equation}
	As $k$ is odd, there is $j\in\{1,\ldots,k\}$ such that $i_j=i_{j+1}$ (here we use the cyclic indexing, i.e. $k+1=1$). By Fubini's theorem
	\begin{equation}\nonumber
	\int\limits_{x_j\in\Omega_{i_j}}\int\limits_{x_{j+1}\in\Omega_{i_{j+1}}}W(x_j,x_{j+1})>0\;,
	\end{equation}
	in other words $\int_{\Omega_i^2}W>0$ for some $i\in\{0,1\}$. As the decomposition $\Omega=\Omega_0\sqcup \Omega_1$ was chosen arbitrarily, this proves that $W$ is not bipartite.
	
	Now suppose that $t(C_k,W)=0$ for every odd integer $k\ge 3$.
	By transfinite induction we define a transfinite sequence $\{(A_{\alpha},B_{\alpha})\colon\alpha\le\gamma\}$ (for some countable ordinal $\gamma$) consisting of pairs of measurable subsets of $\Omega$ such that $\nu(A_{\gamma}\cup B_{\gamma})=1$, and such that for every $\alpha\le\gamma$ the following conditions hold:
	\begin{itemize}
		\item[(i)$_\alpha$] $A_{\beta}\subseteq A_{\alpha}$ and $B_{\beta}\subseteq B_{\alpha}$\; for every $\beta\le\alpha$,
		\item[(ii)$_\alpha$] $\nu(A_{\beta}\cup B_{\beta})<\nu(A_{\alpha}\cup B_{\alpha})$\; for every $\beta<\alpha$,
		\item[(iii)$_\alpha$] $W\restriction_{A_{\alpha}^2}=0$ a.e. and $W\restriction_{B_{\alpha}^2}=0$ a.e.,
		\item[(iv)$_\alpha$] $W\restriction_{(A_{\alpha}\cup B_{\alpha})\times(\Omega\setminus(A_{\alpha}\cup B_{\alpha}))}=0$ a.e.,
		\item[(v)$_\alpha$] $\nu(A_{\alpha}\cap B_{\alpha})=0$.
	\end{itemize}
	Once we are done with the construction, the bipartiteness of $W$ immediately follows by the equation $\nu(A_{\gamma}\cup B_{\gamma})=1$ together with (iii)$_\gamma$ and (v)$_\gamma$.
	
	We start the construction by setting $A_0=B_0=\emptyset$, then the conditions (i)$_\emptyset$-(v)$_\emptyset$ hold trivially. Now suppose that we have already constructed $\{(A_{\alpha},B_{\alpha})\colon\alpha<\alpha_0\}$ for some countable ordinal $\alpha_0$ such that the conditions (i)$_\alpha$-(v)$_\alpha$ hold for every $\alpha<\alpha_0$.
	If $\alpha_0$ is a limit ordinal then we set $A_{\alpha_0}=\bigcup_{\alpha<\alpha_0}A_{\alpha}$ and $B_{\alpha_0}=\bigcup_{\alpha<\alpha_0}B_{\alpha}$, then (i)$_{\alpha_0}$-(v)$_{\alpha_0}$  clearly hold. Otherwise, $\alpha_0=\alpha+1$ for some ordinal $\alpha<\alpha_0$.
	If $\nu(A_\alpha\cup B_{\alpha})=1$ then the construction is finished (with $\gamma:=\alpha$).
	So suppose that $\nu(A_\alpha\cup B_{\alpha})<1$ and denote $\Omega'_\alpha:=\Omega\setminus(A_{\alpha}\cup B_{\alpha})$.
	If $W\restriction_{\Omega'_\alpha\times\Omega'_\alpha}=0$ a.e. then it suffices (by (iv)$_\alpha$) to set $A_{\alpha_0}=A_{\alpha}\cup\Omega'_\alpha$ and $B_{\alpha_0}=B_{\alpha}$. Otherwise there is $x_0\in\Omega'_\alpha$ such that $\nu\left(\{y\in\Omega'_\alpha\colon W(x_0,y)>0\}\right)>0$. By Fubini's theorem, we may also assume that for every odd integer $k\ge 3$ we have that
	\begin{equation}\label{LabledEvenCycle}
	\int\limits_{\Omega^{k-1}}W(x_0,x_1)W(x_1,x_2)\ldots W(x_{k-1},x_0)=0\;.
	\end{equation}
	We set
	\begin{align*}
	D_0&=\{y\in\Omega'_\alpha\colon W(x_0,y)>0\}
	\quad\mbox{, and}\\
	D_{\ell}&=\left\{y\in\Omega'_\alpha\colon\int\limits_{\Omega^{\ell}}W(x_0,x_1)W(x_1,x_2)\ldots W(x_{\ell},y)>0\right\},\;\mbox{for $\ell\ge 1$}\;.
	\end{align*}
	Then we set 
	$$E_\alpha=\bigcup\limits_{\ell\text{ even}}D_{\ell}\quad \mbox{and}
	\quad
	F_\alpha=\bigcup\limits_{\ell\text{ odd}}D_{\ell}\;.$$
	Finally, we define $A_{\alpha_0}=A_{\alpha}\cup E_\alpha$ and $B_{\alpha_0}=B_{\alpha}\cup F_\alpha$. The conditions (i)$_{\alpha_0}$ and (ii)$_{\alpha_0}$ are clearly satisfied, so let us verify only (iii)$_{\alpha_0}$-(v)$_{\alpha_0}$.
	
	As for~(iii)$_{\alpha_0}$, suppose for a contradiction that $W\restriction_{A_{\alpha_0}^2}$ is positive on a set of positive measure.
	By the induction hypothesis, namely by (iii)$_\alpha$ and (iv)$_\alpha$, we easily conclude that $W\restriction_{E_\alpha^2}$ is also positive on a set of positive measure.
	So there are even integers $\ell_1,\ell_2\ge 0$ such that $W\restriction_{D_{\ell_1}\times D_{\ell_2}}$ is positive on a set of positive measure. But then a simple application of Fubini's theorem leads to a contradiction with~(\ref{LabledEvenCycle}) for $k=\ell_1+\ell_2+3$. So we have $W\restriction_{A_{\alpha_0}^2}=0$ a.e. Similarly, we get $W\restriction_{B_{\alpha_0}^2}=0$ a.e. This proves (iii)$_{\alpha_0}$.
	
	As for (iv)$_{\alpha_0}$, suppose for a contradiction that $W\restriction_{(A_{\alpha_0}\cup B_{\alpha_0})\times(\Omega\setminus(A_{\alpha_0}\cup B_{\alpha_0}))}$ is positive on a set of positive measure.
	By the induction hypothesis, namely by (iv)$_\alpha$, $W\restriction_{(E_\alpha\cup F_\alpha)\times(\Omega\setminus(A_{\alpha_0}\cup B_{\alpha_0}))}=W\restriction_{(E_\alpha\cup F_\alpha)\times(\Omega'_\alpha\setminus(E_\alpha\cup F_\alpha))}$ is also positive on a set of positive measure. By Fubini's theorem, there is $z\in\Omega'_\alpha\setminus(E_\alpha\cup F_\alpha)$ such that
	\begin{equation}\nonumber
	\nu\left(\{y\in E_\alpha\cup F_\alpha\colon W(y,z)>0\}\right)>0\;.
	\end{equation}
	So there is an integer $\ell\ge 0$ such that
	\begin{equation}\nonumber
	\nu\left(\{y\in D_{\ell}\colon W(y,z)>0\}\right)>0\;.
	\end{equation}
	By Fubini's theorem, we easily conclude that $z\in D_{\ell+1}$. But this contradicts the fact that $z\notin E_\alpha\cup F_\alpha$.
	
	As for (v)${_{\alpha_0}}$, suppose for a contradiction that $\nu(A_{\alpha_0}\cap B_{\alpha_0})>0$. By the induction hypothesis, namely by (v)$_\alpha$, we also have $\nu(E_\alpha\cap F_\alpha)>0$. So there are an even integer $\ell_1$ and an odd integer $\ell_2$ such that $\nu(D_{\ell_1}\cap D_{\ell_2})>0$. But then an application of Fubini's theorem leads to a contradiction with~(\ref{LabledEvenCycle}) for $k=\ell_1+\ell_2+2$.
	
	To finish the proof, it suffices to observe that for some countable ordinal $\gamma$ we get $\nu(A_{\gamma}\cup B_{\gamma})=1$ (and at that point, the construction stops by the description above). So suppose for a contradiction that this is not the case. Then, by conditions (i)$_\alpha$ and (ii)$_\alpha$, we can find uncountably many pairwise disjoint subsets of $\Omega$ of positive measure, namely $(A_{\alpha+1}\cup B_{\alpha+1})\setminus(A_\alpha\cup B_\alpha)$ (where $\alpha$ is not bounded by any countable ordinal $\gamma$). But this is obviously not possible.
\end{proof}

We can now state Proposition~\ref{prop:morecolors} which generalizes Lemma~\ref{lemma:bipartite} to higher chromatic number. Its proof stems from discussions with Andr\'as M\'ath\'e.

\begin{proposition}\label{prop:morecolors}
	Suppose that $k\in\mathbb N$ and $W:\Omega^2\rightarrow [0,1]$ is a graphon for which $t(H,W)=0$ for each graph of chromatic number $k+1$. Then $W$ is $k$-partite.
\end{proposition}
\begin{proof}
	In the proof, we first approximate $W$ a sequence of samples of $W$-random graphs $G_n\sim\mathbb G(n,W)$. As a second step, we show that with probability one, each $G_n$ is $k$-colorable. Last, we can transfer the $k$-colorings of $G_n$ to a $k$-coloring of $W$. As for the second step, observe that indeed we have
	$$\Prob_{G_n\sim\mathbb G(n,W)}[\text{$G_n$ is not $k$-colorable}]\le \sum_{H}t(H,W)\;,$$where the sum runs over all not $k$-colorable graphs $H$ on $n$ vertices. Since all the terms are~0 by the assumptions of the proposition, we conclude that $G_n$ is indeed $k$-colorable almost surely.
	
	Let $(G_n)_n$ be samples of the inhomogeneous random graphs $\mathbb G(n,W)$. It is well-known that the graphs $G_n$ converge to $W$ in the cut-distance almost surely, see e.g.,~\cite[Lemma~10.16]{Lovasz2012}. We can therefore map the vertices $\{1,2,\ldots,n\}$ of $G_n$ to sets $\Omega_n^{(1)},\Omega_n^{(2)},\ldots,\Omega_n^{(n)}$ which partition $\Omega$ into sets of measure $\frac1n$ each, in a way that this partition witnesses $\epsilon_n$-closeness of $G_n$ to $W$ in the cut-distance (where $\varepsilon_n\rightarrow 0_+$). In particular, whenever $O\subset V(G_n)$ is an independent set, we have
	\begin{equation*}
	\int_{x\in \bigcup_{v\in O}\Omega_n^{(v)}}\int_{y\in \bigcup_{v\in O}\Omega_n^{(v)}}W(x,y)<\epsilon_n\;.
	\end{equation*}
	
	Recall $G_n$ is $k$-colorable (with probability one). Let us fix a partition $$V_n^{(1)}\sqcup V_n^{(2)}\sqcup \ldots \sqcup V_n^{(k)}=\Omega$$ according to one fixed $k$-coloring of $G_n$. 
	
	Consider now a weak$^*$ accumulation point $(f_1,f_2,\ldots,f_k)$ of the sequence of $k$-tuples of functions 
	$$\left(\mathbf 1_{V_n^{(1)}} \:,\:\mathbf 1_{V_n^{(2)}}\:,\ldots\:,\: \mathbf 1_{V_n^{(k)}}\right)_n\;.$$
	Such an accumulation point exists by the sequential Banach--Alaoglu theorem, see Section~\ref{ssec:WeakConv}.
	
	We have $f_1+f_2+\ldots+f_k=1$ almost everywhere on $\Omega$. Consequently we can find measurable sets $A_i\subset \SUPPORT f_i$ so that $\Omega=A_1\sqcup A_2\sqcup\ldots\sqcup A_k$. Lemma~\ref{lem:limitofalmostindependent} tells us that $A_1\sqcup A_2\sqcup\ldots\sqcup A_k$ is a $k$-coloring of $W$.
\end{proof}

Let us note that in~\cite{DHM:CliquesRandom} a result in a similar direction was proven:
\begin{theorem}
	Suppose that $W$ is a graphon, and $H$ is a finite graph with the property
	that $t(H, W ) = 0$. Then $W$ is countably partite.
\end{theorem}

\subsubsection{A technical lemma}
Lemma~\ref{lemma:bipartite} tells us that if a graphon is not bipartite, then it has a positive density of some odd cycle $C_k$. The next lemma allows us to zoom in into some location, where some of these $C_k$'s are very densely located. That is, we will find sets $A_1,\ldots,A_k\subset \Omega$ such that $W$ is positive on most of $A_h\times A_{h+1}$ for each $h\in[k]$. The additional properties, which we will need for our proof of Theorem~\ref{thm:CovIntegralImpliesBipartite} later, are that the sets $A_h$ have the same measure, and are disjoint.
\begin{lemma}\label{lem:oddcycles}
	Suppose that $W:\Omega^2\rightarrow [0,1]$ is a graphon. If $W$ is not bipartite then there exists an odd integer $k\ge 3$ with the following property. For each $\epsilon>0$ there exist pairwise disjoint sets $A_1,\ldots,A_k\subset \Omega$ of the same positive measure $\alpha$, such that for each $h\in[k]$, $W$ is positive everywhere on $A_{h}\times A_{h+1}$ except a set of measure at most $\epsilon\alpha^2$. Here, we use cyclic indexing, $A_{k+1}=A_1$.
\end{lemma}

\begin{proof}
	Suppose that $W$ is not bipartite. By Lemma~\ref{lemma:bipartite} there is an odd integer $k\ge 3$ such that
	\begin{equation}\label{eq:t}
	t:=\int\limits_{\Omega^k}W(x_1,x_2)W(x_2,x_3)\ldots W(x_k,x_1)>0\;.
	\end{equation}
	We find a natural number $n$ such that
	\begin{equation}\label{eq:n}
	n>\frac{k^2}{t}\;.
	\end{equation}
	We fix a decomposition $\Omega=\bigcup\limits_{i=1}^n\Omega_i$ of $\Omega$ into pairwise disjoint sets of the same measure $\tfrac{1}{n}$.
	We also set
	\begin{equation}\label{eq:D}
	\begin{split}
	D=\big\{(x_1,\ldots,x_k)\in\Omega^k\colon&\text{there are }i,j\in\{1,\ldots,k\}\text{ and }\ell\in\{1,\ldots,n\}\\
	&\text{ such that }i\neq j\text{ and }x_i,x_j\in\Omega_{\ell}\big\}\;.
	\end{split}
	\end{equation}
	Then we have
	\begin{equation}\nonumber
	\nu^{\otimes k}\left(D\right)\le\sum\limits_{\substack{i,j=1,\ldots,k\\ i\neq j}}\frac 1n\le\frac{k^2}{n}\stackrel{(\ref{eq:n})}{<}t\;,
	\end{equation}
	and so
	\begin{equation}\label{eq:IntegralOverDiagonal}
	\int\limits_DW(x_1,x_2)W(x_2,x_3)\ldots W(x_k,x_1)\le\nu^{\otimes k}(D)<t\;.
	\end{equation}
	By (\ref{eq:t}) and (\ref{eq:IntegralOverDiagonal}), we get
	\begin{equation}\nonumber
	\int\limits_{\Omega^k\setminus D}W(x_1,x_2)W(x_2,x_3)\ldots W(x_k,x_1)>0\;.
	\end{equation}
	By this and (\ref{eq:D}) there are pairwise distinct integers $\ell_1,\ldots,\ell_k\in\{1,\ldots,n\}$ such that
	\begin{equation}
	\int\limits_{\Omega_{\ell_1}}\int\limits_{\Omega_{\ell_2}}\cdots\int\limits_{\Omega_{\ell_k}}W(x_1,x_2)W(x_2,x_3)\ldots W(x_k,x_1)>0\;,
	\end{equation}
	and so the set
	\begin{equation}\label{eq:definitionOfE}
	E:=\{(x_1,x_2,\ldots,x_k)\in\Omega_{\ell_1}\times\Omega_{\ell_2}\times\ldots\times\Omega_{\ell_k}\colon W(x_1,x_2)W(x_2,x_3)\ldots W(x_k,x_1)>0\}
	\end{equation}
	is of positive measure.
	
	Now let us fix $\varepsilon>0$, and let $\delta>0$ be such that
	\begin{equation}\label{eq:delta}
	\frac{\nu^{\otimes k}(E)-\delta}{\nu^{\otimes k}(E)+\delta}\ge 1-\frac{\eps}{2}\;.
	\end{equation}
	Recall that the $\sigma$-algebra of all measurable subsets of $\Omega_{\ell_1}\times\ldots\times\Omega_{\ell_k}$ is generated by the algebra consisting of all finite unions of measurable rectangles.
	Thus there is a finite union $S=\bigcup\limits_{i=1}^mR_i$ of measurable rectangles $R_1,\ldots,R_m$ in $\Omega_{\ell_1}\times\ldots\times\Omega_{\ell_k}$ such that
	\begin{equation}\nonumber
	\nu^{\otimes k}(E\triangle S)=\nu^{\otimes k}(E\setminus S)+\nu^{\otimes k}(S\setminus E)\le\delta\;.
	\end{equation}
	Without loss of generality, we may assume that the measurable rectangles $R_1,\ldots,R_m$ are pairwise disjoint.
	Then we have
	\begin{equation}\label{eq:pomer}
	\frac{\nu^{\otimes k}(E\cap S)}{\nu^{\otimes k}(S)}\ge\frac{\nu^{\otimes k}(E)-\nu^{\otimes k}(E\setminus S)}{\nu^{\otimes k}(E)+\nu^{\otimes k}(S\setminus E)}\ge\frac{\nu^{\otimes k}(E)-\delta}{\nu^{\otimes k}(E)+\delta}\stackrel{(\ref{eq:delta})}{\ge} 1-\frac{\eps}{2}\;.
	\end{equation}
	Now the left-hand side of (\ref{eq:pomer}) can be expressed as
	\begin{equation}\nonumber \frac{\nu^{\otimes k}(E\cap S)}{\nu^{\otimes k}(S)}=\sum\limits_{i=1}^m\frac{\nu^{\otimes k}(R_i)}{\nu^{\otimes k}(S)}\cdot\frac{\nu^{\otimes k}(E\cap R_i)}{\nu^{\otimes k}(R_i)}\;,
	\end{equation}
	i.e.\ as a convex combination of $\frac{\nu^{\otimes k}(E\cap R_i)}{\nu^{\otimes k}(R_i)}$, $i=1,\ldots,m$.
	Therefore by (\ref{eq:pomer}), there is an index $i_0\in\{1,\ldots,m\}$ such that
	\begin{equation}\label{eq:R_{i_0}}
	\frac{\nu^{\otimes k}(E\cap R_{i_0})}{\nu^{\otimes k}(R_{i_0})}\ge 1-\frac{\eps}{2}\;.
	\end{equation}
	Let $R_{i_0}$ be of the form $R_{i_0}=B_1\times\ldots\times B_k$.
	Find a natural number $p$ such that
	\begin{equation}\label{eq:p}
	p\ge\frac{2k}{\eps\nu^{\otimes k}(R_{i_0})}\;.
	\end{equation}
	For every $i=1,\ldots,k$, we fix a finite decomposition $B_i=B_i^0\cup\bigcup\limits_{j=1}^{q_i}B_i^j$ of $B_i$ into $1+q_i$ many pairwise disjoint sets, such that $q_i:=\lfloor p\nu(B_i)\rfloor$, $\nu(B_i^0)\le\tfrac{1}{p}$ and $\nu(B_i^j)=\tfrac{1}{p}$ for $j=1,\ldots,q_i$.
	Then we clearly have
	\begin{equation}\label{eq:zarovnani}
	\nu^{\otimes k}\left(R_{i_0}\setminus\prod\limits_{i=1}^k\bigcup\limits_{j=1}^{q_i}B_i^j\right)\le\frac{k}{p}\;,
	\end{equation}
	and so
	\begin{equation}\label{eq:PoSeriznuti}
	\frac{\nu^{\otimes k}\left(E\cap\prod\limits_{i=1}^k\bigcup\limits_{j=1}^{q_i}B_i^j\right)}{\nu^{\otimes k}\left(\prod\limits_{i=1}^k\bigcup\limits_{j=1}^{q_i}B_i^j\right)}\stackrel{(\ref{eq:zarovnani})}{\ge}\frac{\nu^{\otimes k}(E\cap R_{i_0})-\tfrac{k}{p}}{\nu^{\otimes k}(R_{i_0})}\stackrel{(\ref{eq:R_{i_0}})}{\ge}1-\frac{\eps}{2}-\frac{k}{p\nu^{\otimes k}(R_{i_0})}\stackrel{(\ref{eq:p})}{\ge}1-\eps\;.
	\end{equation}
	The left-hand side of (\ref{eq:PoSeriznuti}) can be expressed as the following convex combination:
	\begin{equation}\nonumber
	\frac{\nu^{\otimes k}\left(E\cap\prod\limits_{i=1}^k\bigcup\limits_{j=1}^{q_i}B_i^j\right)}{\nu^{\otimes k}\left(\prod\limits_{i=1}^k\bigcup\limits_{j=1}^{q_i}B_i^j\right)}=\sum\limits_{j_1=1}^{q_1}\cdots\sum\limits_{j_k=1}^{q_k}\frac{\nu^{\otimes k}\left(\prod\limits_{i=1}^kB_i^{j_i}\right)}{\nu^{\otimes k}\left(\prod\limits_{i=1}^k\bigcup\limits_{j=1}^{q_i}B_i^j\right)}\cdot\frac{\nu^{\otimes k}\left(E\cap\prod\limits_{i=1}^kB_i^{j_i}\right)}{\nu^{\otimes k}\left(\prod\limits_{i=1}^kB_i^{j_i}\right)}\;.
	\end{equation}
	Therefore by (\ref{eq:PoSeriznuti}), there are $j_i\in\{1,\ldots,q_i\}$, $i=1,\ldots,k$, such that
	\begin{equation}\nonumber
	\frac{\nu^{\otimes k}\left(E\cap\prod\limits_{i=1}^kB_i^{j_i}\right)}{\nu^{\otimes k}\left(\prod\limits_{i=1}^kB_i^{j_i}\right)}\ge 1-\eps\;,
	\end{equation}
	or equivalently
	\begin{equation}\label{eq:equivalently}
	\frac{\nu^{\otimes k}\left(\prod\limits_{i=1}^kB_i^{j_i}\setminus E\right)}{\nu^{\otimes k}\left(\prod\limits_{i=1}^kB_i^{j_i}\right)}\le\eps\;.
	\end{equation}
	We set $A_i=B_i^{j_i}$ for $i=1,\ldots,k$.
	Then $A_1,\ldots,A_k$ are pairwise disjoint (as $A_i\subseteq B_i\subseteq\Omega_{\ell_i}$ for every $i$), and each of these sets has the same measure $\alpha=\tfrac{1}{p}$.
	
	By~(\ref{eq:definitionOfE}) we have 
	\begin{equation}\label{eq:poui}
	\left\{(x_1,x_2,\ldots,x_k)\in A_1\times A_2\times\ldots\times A_{k}\colon W(x_1,x_2)W(x_2,x_3)\ldots W(x_k,x_1)=0\right\}=\prod\limits_{i=1}^kB_i^{j_i}\setminus E\;.
	\end{equation}
	By~(\ref{eq:equivalently}),
	$$\nu^{\otimes k}\left(\prod\limits_{i=1}^kB_i^{j_i}\setminus E\right)\le\epsilon \nu^{\otimes k}\left(\prod\limits_{i=1}^kB_i^{j_i}\right)=\frac{\epsilon}{p^k}\;.$$
	For each $h\in[k]$, we have $$\prod\limits_{i=1}^kB_i^{j_i}\setminus E\supseteq A_1\times\ldots\times A_{h-1}\times \left\{(x_h,x_{h+1})\in A_h\times A_{h+1}\colon W(x_h,x_{h+1})=0\right\}\times A_{h+2}\times\ldots A_k\;.$$
	Plugging this into~\eqref{eq:poui}, we get
	\begin{equation}\nonumber
	\nu^{\otimes 2}\left(\left\{(x_h,x_{h+1})\in A_h\times A_{h+1}\colon W(x_h,x_{h+1})=0\right\}\right)\le\frac{\eps}{p^2}\;,
	\end{equation}
	as required.
\end{proof}

\subsubsection{Application: spectra of bipartite graphons}\label{sssection:spectrabipartite}
It is a well-known fact that a finite graph is bipartite if and only if the spectrum of its adjacency matrix is symmetric. In Theorem~\ref{thm:spectrumbipartite}, we prove a counterpart of this fact for graphons. This result seems to be new. Indeed, just as in the finite case, to obtain this result one needs the characterization of bipartite graphs using odd cycles, which we gave in Lemma~\ref{lemma:bipartite}. This section is not needed for our main results concerning matching polytons.

Let us briefly recall the notion of eigenvalues of graphons, following~\cite[Section 7.5]{Lovasz2012}, where more details can be found. To a given graphon $W:\Omega^2\rightarrow [0,1]$ we can associate a kernel operator $T_W:\LpSpace^2(\Omega)\rightarrow\LpSpace^2(\Omega)$,
$$\left(T_W(f)\right)(x):=\int_{y\in\Omega} W(x,y)f(y)\;.$$
Then $T_W$ is a Hilbert--Schmidt operator, and thus it has a real spectrum of finitely or countably many nonzero eigenvalues. The spectrum is said to be \emph{symmetric} if each $\lambda\in\mathbb{R}$ is an eigenvalue if and only if $-\lambda$ is, and their multiplicities are the same.
\begin{theorem}\label{thm:spectrumbipartite}
A graphon is bipartite if and only if the associated kernel operator has a symmetric spectrum.
\end{theorem}
\begin{proof}
Suppose first that $W:\Omega^2\rightarrow [0,1]$ is a bipartite graphon, and the colour classes are $A\sqcup B=\Omega$. Let $\lambda\in\mathbb{R}$ be an arbitrary eigenvalue and let $f\in L^2(\Omega)$ be the corresponding eigenfunction. Define a function $g\in L^2(\Omega)$ by setting $g(x)=f(x)$ for $x\in A$ and $g(x)=-f(x)$ for $x\in B$. Now, for almost every $x\in A$ we have
\begin{align*}
T_W(g)(x)&=\int_{y\in\Omega} W(x,y)g(y)=\int_{y\in B} W(x,y)\cdot(-f(y))=-\int_{y\in B} W(x,y) f(y)\\
&=-\int_{y\in \Omega} W(x,y) f(y)=-T_W(f)(x)=-\lambda f(x)=-\lambda g(x)\;,
\end{align*}
where the fact that $A$ is an independent set justifies the second and the fourth equality, and the fact that $f$ is an eigenvalue of $T_W$ is used for the sixth equality. Similar calculations give that for almost every $x\in B$ we have that $T_W(g)(x)=-\lambda g(x)$. We conclude that $g$ is an eigenfunction for eigenvalue $-\lambda$. Hence, $T_W$ has a symmetric spectrum.

Let us now prove the converse direction. Suppose that $W$ is a graphon such that $T_W$ has a symmetric spectrum. In particular, for every  $k\in\{3,5,7,\ldots\}$ we have that $\sum_{\lambda\in \mathrm{Spec}(T_W)}\lambda^k=0$. Just as in the graph case, the sum of the $k$-th powers of the eigenvalues corresponds to the $k$-cycle density (see~\cite[Equation~(7.22)]{Lovasz2012}), i.e., $t(C_k,W)=0$. Lemma~\ref{lemma:bipartite} then tells us that $W$ is bipartite.
\end{proof}

\subsection{Introducing matchings and vertex covers in graphons}\label{sec:matchings}
We introduce the notion of matchings in a graphon. Our definitions follow~\cite{HlHuPi:TilingsInGraphons}, where they were given in the more general context of $F$-tilings; for completeness, we recall this more general notation in Section~\ref{sssec:Ftilings}.
\begin{definition}\label{def:matching}
Suppose that $W:\Omega^2\rightarrow[0,1]$ is a graphon. We say that a function $\mathfrak m\in\LpSpace^1(\Omega^2)$ is a \emph{matching} in $W$ if
\begin{enumerate}[label=(M\arabic*)]
\item $\mathfrak m\ge 0$ almost everywhere,
\item $\SUPPORT \mathfrak m\subset \SUPPORT W$ up to a null-set, and
\item for almost every $x\in\Omega$, we have $\int_y \mathfrak m(x,y)+\int_y \mathfrak m(y,x)\le 1$.
\end{enumerate}
\end{definition}
In~\cite{HlHuPi:TilingsInGraphons} it is argued in detail why this is ``the right'' notion of matchings. We do not repeat this discussion here and only make two comments. Firstly, the requirements in Definition~\ref{def:matching} are counterparts to fractional matchings in finite graphs. Namely, a fractional matching in a graph $G$ can be represented as a function $f:V(G)^2\rightarrow \mathbb R$ such that
\begin{enumerate}[label=(F\arabic*)]
\item\label{enu:finitematching1} $f\ge 0$,
\item\label{enu:finitematching2} if $f(x,y)>0$ then $xy\in E(G)$, and
\item\label{enu:finitematching3} for every $x\in V(G)$, we have $\sum_y f(x,y)+\sum_y f(y,x)\le 1$.
\end{enumerate}
Note that usually fractional matchings are represented using symmetric functions. This is however only a notational matter. The current choice for these functions being not-necessarily symmetric is adopted from~\cite{HlHuPi:TilingsInGraphons}. (There, this choice was dictated not by matching, but rather by creating a general concept of $F$-tilings even for graphs which are not vertex-transitive.) 

Secondly, note that the actual values of $W$ do not play any role in Defition~\ref{def:matching}, only the support of $W$ matters. To explain this to a curious reader, we need to assume some familiarity with the regularity method, which is a finite counterpart of the theory of graphons. (Readers who are not familiar with the regularity method or not curious may skip this text.) The fact that the values of $W$ are not relevant (as long as they are positive) reflects the situation in the blow-up lemma~\cite{Komlos1997}, which says that one can find a perfect matching in any super-regular pair of any density (as long as it is positive).

\medskip

Given a matching $\mathfrak m$ in a graphon $W$ we define its \emph{size}, $\|\mathfrak m\|=\int_x\int_y \mathfrak m(x,y)$. The \emph{matching ratio} of $W$, denoted by $\matchratio(W)$, is defined as the supremum of the sizes of all matchings in $W$.

\begin{remark}\label{rem:intVSfrac}
	As said already in the Introduction, even though Definition~\ref{def:matching} is inspired by fractional matchings in finite graphs, the resulting graphon concept is referred to as ``matchings''. This is because in the graphon world  every function $\mathfrak m$ from Definition~\ref{def:matching} behaves in many ways as an integral matching. One demonstration of this is Theorem~\ref{thm:tilingtransference} which relates sizes of matchings in finite graphs to the matching ratio of the limit graphon. 
	
	On the other hand, for a graphon representation $W_G$ of one finite graph $G$ there is correspondence between matchings in $W_G$ and integral matchings in $G$; see Section~\ref{sssec:polytonsofrepresentations}. Thus, the previous paragraph applies only to the limit setting.
\end{remark}

We write $\matchingpolytope(W)\subset \LpSpace^1(\Omega^2)$ for the set of all matchings in $W$. It is straightforward to check that this set is convex (like the set of fractional matchings in a finite graph) and closed (if we consider the norm topology on $\LpSpace^1(\Omega^2)$). But --- unlike the finite case --- it need not be compact. To see this consider the graphon $U:[0,1]^2\rightarrow[0,1]$ defined as $U(x,y)=1$ for $x+y\le 1$ and $U(x,y)=0$ for $x+y>1$. This example was first given in~\cite{HlHuPi:TilingsInGraphons} in a somewhat different context. For $\epsilon$ positive, consider a matching $\mathfrak m_\epsilon$ defined to be $1/(2\epsilon)$ on a stripe of width $\epsilon$ along the diagonal $x+y=1$ and zero otherwise. This is shown on Figure~\ref{fig:half}. It is clear that the matchings $\mathfrak m_\epsilon$ do not contain any convergent subsequence, as we let $\epsilon\rightarrow 0_+$. Considering the weak topology on the space $\LpSpace^1(\Omega^2)$ (that is the topology generated by the dual space $\LpSpace^{\infty}(\Omega^2)$) does not help as the same counterexample easily shows. Therefore considering the set $\matchingpolytope(W)$ as a subset of the second dual of $\LpSpace^1(\Omega^2)$ equipped with its weak$^*$ topology seems to be the only reasonable way to have a natural compactification of $\matchingpolytope(W)$. However, we did not need go that far.
\begin{figure}
\includegraphics[scale=0.8]{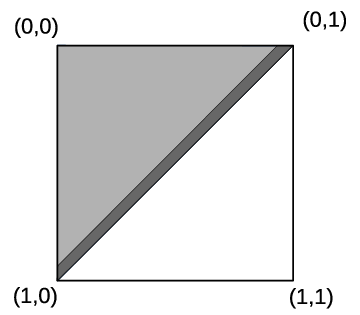} 
\caption{The graphon $U$ discussed in the text. A matching $\mathfrak m_\epsilon$ shown in dark gray. The support of $U$ is shown in light gray.}
\label{fig:half}
\end{figure}

\bigskip
We can now proceed with the definition of fractional vertex covers of a graphon. First, recall that a function $c:V(G)\rightarrow[0,1]$ is a \emph{fractional vertex cover} of a finite graph $G$ if we have $c(x)+c(y)\ge 1$ for each $xy\in E(G)$. Thus, the graphon counterpart is as follows.
\begin{definition}\label{def:cover}
Suppose that $W:\Omega^2\rightarrow[0,1]$ is a graphon. We say that a function $\mathfrak c\in\LpSpace^\infty(\Omega)$ is a \emph{fractional vertex cover} of $W$ if $0\le \mathfrak c\le 1$ almost everywhere and the set $$\SUPPORT W\setminus\{(x,y):\mathfrak c(x)+\mathfrak c(y)\ge 1\}$$
has measure 0.
\end{definition}
A fractional vertex cover is called an \emph{integral vertex cover} if its values are from the set $\{0,1\}$ almost everywhere.

Given a fractional vertex cover $\mathfrak c$ of a graphon $W$ we define its \emph{size}, $\|\mathfrak c\|=\int_x \mathfrak c(x)$. The fractional cover number, $\fraccoverratio(W)$ is the infimum of sizes of all fractional vertex covers of $W$.

We write $\fraccoverpolytope(W)\subset \LpSpace^\infty(\Omega)$ for the set of all fractional vertex covers of $W$. It is straightforward to check that this set is convex. Further, as was first shown in~\cite[Theorem~3.14]{HlHuPi:TilingsInGraphons}, it is also compact in the space $\LpSpace^\infty(\Omega)$ equipped with the weak$^*$ topology.

\subsubsection{$F$-tilings}\label{sssec:Ftilings} In this brief section we show, how the concept of matchings is generalized in~\cite{HlHuPi:TilingsInGraphons} to general $F$-tilings. This section is just for a comparison and these more general definitions are not needed in the rest of the paper. Let us recall that given a graph $F$, an \emph{$F$-tiling} in a graph $G$ is a collection of (not necessarily induced) vertex-disjoint copies of $F$ in $G$. Let us now include a corresponding definition for graphons.
\begin{definition}
	\label{def:graphontiling}Suppose that $W:\Omega^{2}\rightarrow[0,1]$
	is a graphon, and that $F$ is a graph on the vertex set $[k]$. A function $\mathfrak{t}\in L^1(\Omega^{k})$ is called 	an \emph{$F$-tiling} in $W$ if 
\begin{enumerate}
\item $\mathfrak{t}\ge 0$,
\item $\SUPPORT \mathfrak{t}\subset \left\{(x_1,\ldots,x_k)\in \Omega^k:\prod_{i,j:ij\in E(F)}W(x_i,x_j)>0\right\}$, and
\item  for each $\ell\in [k]$ and for almost every $x_\ell \in\Omega$, we have that 
$ \sum_{\ell=1}^{k}\int_{x_{1},\ldots x_{\ell-1},x_{\ell+1},\ldots,x_{k}}\mathfrak{t}(x_{1},\ldots,x_{k})\;\mathrm{d}\nu^{\otimes k-1}\le 1$.
\end{enumerate}
	The \emph{size} of an $F$-tiling $\mathfrak{t}$ is $\|\mathfrak{t}\|=\int\mathfrak{t}(x_{1},\ldots,x_{k})\;\mathrm{d}\nu^{k}$.
	The \emph{$F$-tiling number} of $W$ is
	the supremum of sizes over all $F$-tilings in $W$.
\end{definition}

\subsubsection{Matchings and fractional vertex covers of representations of finite graphs}\label{sssec:polytonsofrepresentations} 

Suppose that $G$ is a graph, and $W_G:\Omega^2\rightarrow[0,1]$ is a graphon representation of $G$ corresponding to a partition $\bigsqcup_{v\in V(G)}\Omega_v=\Omega$. Given a function $\mathfrak m\in\LpSpace^1(\Omega^2)$, define a function $f_{\mathfrak m}:V(G)^2\rightarrow \mathbb{R}$ by
$$f_{\mathfrak m}(u,v):=v(G)^2\cdot\int_{\Omega_u}\int_{\Omega_v}\mathfrak m\;.$$
It is straightforward to check that if $\mathfrak m$ is a matching in $W_G$ in the sense of Definition~\ref{def:matching} then $f_{\mathfrak m}$ is a fractional matching in $G$ in the sense of~\ref{enu:finitematching1}--\ref{enu:finitematching3}. The converse is not true in general: it may happen than $f_{\mathfrak m}$ is a fractional matching in $G$ but $\mathfrak m$ does not satisfy either (M1) or (M3) (or both) from Definition~\ref{def:matching}. However, it is easy to verify that we have
\begin{equation}
\begin{split}
\matchingpolytope(W_G)=\{\mathfrak m\in L^1(\Omega^2)\colon \mathfrak m\ge 0\text{ a.e.}, \int_y\mathfrak m(x,y)+\int_y\mathfrak m(y,x)\le 1\text{ for a.e. }x\in\Omega,\\ f_\mathfrak m\text{ is a fractional matching in }G \}\;.
\end{split}
\end{equation}

Given a function $\mathfrak c:\Omega\rightarrow[0,1]$, define a function $g_{\mathfrak c}:V(G)\rightarrow [0,1]$ by
\begin{equation}\label{eq:gc}
g_{\mathfrak c}(u):=\mathrm{essinf}\left(\mathfrak c_{\restriction{\Omega_u}}\right)\;.
\end{equation}
It is straightforward to check that $\mathfrak c$ is a fractional vertex cover in $W_G$ in the sense of Definition~\ref{def:cover} if and only if $g_{\mathfrak c}$ is a fractional vertex cover in $G$ in the usual sense. 

Let us try to reverse this: for a function $g:V(G)\rightarrow [0,1]$, let $\mathfrak c_g:\Omega\rightarrow [0,1]$ be defined by
\begin{equation}\label{eq:defc}
\mathfrak c_g(x):=c(u)\quad \text{for each $u\in V(G)$ and each $x\in \Omega_u$}\;.
\end{equation}
Then one can easily verify that
\begin{equation}
\begin{split}
\fraccoverpolytope(W_G)=\{\mathfrak c\in L^{\infty}(\Omega)\colon \text{ there is a fractional vertex cover } g \text{ of } G\\ \text{ such that } \mathfrak c_g(x)\le\mathfrak c(x)\le 1 \text{ for a.e. }x\in\Omega \}\;.
\end{split}
\end{equation}


\section{Extreme points of fractional vertex cover polytons}\label{sec:extremepoints}
\subsection{Digest of properties of fractional vertex cover polytopes}\label{sub:fvcpFiniteDigest}
Let $G$ be a finite graph. Then the set $\fraccoverpolytope(G)\subset [0,1]^{V(G)}$ is a polytope. It is a fundamental fact in combinatorial optimization that all the vertices of this polytope are half-integral (i.e., in the form $\{0,\frac12,1\}^{V(G)}$), \cite[Theorem 30.2]{Schrijver2003}. Furthermore, all its vertices are integral if and only if $G$ is bipartite, \cite[Theorem 18.3]{Schrijver2003}. 

\subsection{Graphon counterparts}\label{sub:graphonintegrality}
In view of the Krein--Milman Theorem (see Section~\ref{sssec:KreinMilman}) the graphon counterparts to the results described in Section~\ref{sub:fvcpFiniteDigest} will be expressed in terms of $\calE(\fraccoverpolytope(W))$. Let us now state these counterparts. We say that a function $f\in \LpSpace^\infty(\Omega)$ is \emph{integral} if for almost all $x\in\Omega$ we have $f(x)\in\{0,1\}$. We say that a function $f\in \LpSpace^\infty(\Omega)$ is \emph{half-integral} if for almost all $x\in\Omega$ we have $f(x)\in\{0,\frac 12,1\}$.

The following three theorems (proven later) are the main results of this section.

\begin{theorem}\label{thm:CovHalfIntegral}
Suppose that $W:\Omega^2\rightarrow[0,1]$ is a graphon. Then all the extreme points of 
$\fraccoverpolytope(W)$ are half-integral.
\end{theorem}
\begin{theorem}\label{thm:CovIntegral}
	Suppose that $W:\Omega^2\rightarrow[0,1]$ is a bipartite graphon. Then all the extreme points of $\fraccoverpolytope(W)$ are integral.
\end{theorem}
\begin{theorem}\label{thm:CovIntegralImpliesBipartite}
Suppose that $W:\Omega^2\rightarrow[0,1]$ is a graphon. If all the extreme points of 
$\fraccoverpolytope(W)$ are integral then $W$ is bipartite.
\end{theorem}
Observe that every extreme point of the fractional vertex cover polyton of a bipartite graphon is in fact exposed. Indeed, let $\phi\in \calE(\fraccoverpolytope(W))$ for some bipartite graphon $W:\Omega^2\rightarrow [0,1]$. Then Theorem~\ref{thm:CovIntegral} tells us that the sets $A=\phi^{-1}(0)$ and $B=\phi^{-1}(1)$ partition $\Omega$ up to a 0-measure set. It is now clear that the linear functional $f$,
$$f(\psi):=\int_B \psi(x)\mathrm d x-\int_A \psi(x)\mathrm d x$$ is strictly maximized at $\phi$ on $\fraccoverpolytope(W)$. We do not know whether the same can be proved for non-bipartite graphons as well (see Section~\ref{exposed}).

\subsection{Proof of Theorem~\ref{thm:CovIntegral}}
During the proof of Theorem~\ref{thm:CovIntegral}, we employ the notation $\partial(z):=\min(z,1-z)$ for $z\in[0,1]$. We shall need the following easy fact.
\begin{fact}\label{fact:perturbe}
Suppose that $x,y\in[0,1]$ are two reals which satisfy $x+y\ge 1$. Then the numbers $x^+=x+\partial(x)$ and $y^-=y-\partial(y)$ satisfy $x^+,y^-\in[0,1]$ and $x^+ + y^-\ge 1$.
\end{fact}
\begin{proof}
The fact that $x^+,y^-\in[0,1]$ is obvious. To prove that $x^+ + y^-\ge 1$, we distinguish three cases. First, suppose that $0\le x\le \frac12\le y\le 1$. Then $\partial(x)=x$, $\partial(y)=1-y$, and consequently, $x^+ + y^-=2x+2y-1\ge 1$. Second, suppose that $0\le y\le \frac12\le x\le 1$. Then $\partial(x)=1-x$, $\partial(y)=y$, and consequently, $x^+ + y^-=1$. Third, suppose that $\frac12\le x,y\le 1$. Then $\partial(x)=1-x$, $\partial(y)=1-y$, and consequently, $x^+ + y^-=2y\ge 1$.
\end{proof}

\begin{proof}[Proof of Theorem~\ref{thm:CovIntegral}]
Let $\Omega=\Omega_A\sqcup\Omega_B$ be a partition into two sets of positive measure such that $W$ is zero almost everywhere on $(\Omega_A\times\Omega_A)\cup(\Omega_B\times\Omega_B)$.
Suppose that $\frakc\in\fraccoverpolytope(W)$ is not integral. Using the notation from Fact~\ref{fact:perturbe}, define two functions $\frakc',\frakc'':\Omega_A\sqcup\Omega_B\rightarrow \bbR$ by $\frakc'(a)=\frakc(a)^+$, $\frakc'(b)=\frakc(b)^-$, $\frakc''(a)=\frakc(a)^-$, $\frakc''(b)=\frakc(b)^+$, for each $a\in\Omega_A$ and $b\in\Omega_B$. By Fact~\ref{fact:perturbe}, we have that $\frakc',\frakc''\in\fraccoverpolytope(W)$. Further $\frakc=\frac12(\frakc'+\frakc'')$. As $\frakc$ is not integral, we have that $\frakc$ is distinct from $\frakc'$ and $\frakc''$. We conclude that $\frakc$ is not an extreme point of $\fraccoverpolytope(W)$.
\end{proof}

\subsection{Proof of Theorem~\ref{thm:CovHalfIntegral}} The proof of Theorem~\ref{thm:CovHalfIntegral} is very similar to that of Theorem~\ref{thm:CovIntegral}.
We first state the counterpart of Fact~\ref{fact:perturbe} we need to this end. We omit the proof as it is almost the same as that of Fact~\ref{fact:perturbe}.
Here, we employ the notation $\partial^{\bullet}(z):=\min(z,1-z,|\frac12-z|)$ for $z\in[0,1]$.

\begin{fact}\label{fact:perturbehalf}
Suppose that $x,y\in[0,1]$ are two reals which satisfy $x+y\ge 1$. Then the numbers $x^{+\bullet}=x+\partial^{\bullet}(x)$ and $y^{-\bullet}=y-\partial^{\bullet}(y)$ satisfy $x^{+\bullet},y^{-\bullet}\in[0,1]$ and $x^{+\bullet} + y^{-\bullet}\ge 1$.
\end{fact}
\begin{proof}[Proof of Theorem~\ref{thm:CovHalfIntegral}]
Suppose that $\frakc\in\fraccoverpolytope(W)$ is not half-integral. Consider the sets $\Omega_A=\{x\in\Omega:0\le\frakc(x)\le\frac12\}$ and $\Omega_B=\{x\in\Omega:\frac12<\frakc(x)\le1\}$. Using the notation from Fact~\ref{fact:perturbehalf}, define two functions $\frakc',\frakc'':\Omega_A\sqcup\Omega_B\rightarrow \bbR$ by $\frakc'(a)=\frakc(a)^{+\bullet}$, $\frakc'(b)=\frakc(b)^{-\bullet}$, $\frakc''(a)=\frakc(a)^{-\bullet}$, $\frakc''(b)=\frakc(b)^{+\bullet}$, for each $a\in\Omega_A$ and $b\in\Omega_B$. By Fact~\ref{fact:perturbehalf}, we have that $\frakc',\frakc''\in\fraccoverpolytope(W)$. Further $\frakc=\frac12(\frakc'+\frakc'')$. As $\frakc$ is not half-integral, we have that $\frakc$ is distinct from $\frakc'$ and $\frakc''$. We conclude that $\frakc$ is not an extreme point of $\fraccoverpolytope(W)$.
\end{proof}

\subsection{Proof of Theorem~\ref{thm:CovIntegralImpliesBipartite}}

\begin{proof}[Proof of Theorem~\ref{thm:CovIntegralImpliesBipartite}]
We shall prove the counterpositive. Suppose that $W$ is not bipartite. Let $\coverpolytope(W)$ be the closure (in the weak$^*$ topology) of the convex hull of all integral vertex covers of $W$. Clearly, we have $\coverpolytope(W)\subset\fraccoverpolytope(W)$, and each integral vertex cover of $\fraccoverpolytope(W)$ is contained in $\coverpolytope(W)$. Below, we shall show that  
\begin{equation}\label{eq:fracVSIntNonempty}
\fraccoverpolytope(W)\setminus \coverpolytope(W)\neq\emptyset\;.
\end{equation}
The Krein--Milman Theorem (see Section~\ref{sssec:KreinMilman}) then tells us that $\calE\big(\fraccoverpolytope(W)\big)\setminus \coverpolytope(W)\neq\emptyset$. It will thus follow that there exists a non-integral fractional vertex cover in $\calE\big(\fraccoverpolytope(W)\big)$, as was needed to show.

Take $\frakc:\Omega\rightarrow [0,1]$ to be constant $\frac12$. Clearly, $\frakc\in\fraccoverpolytope(W)$. In order to show~\eqref{eq:fracVSIntNonempty}, it suffices to prove that $\frakc\notin\coverpolytope(W)$. Let $k$ be the odd integer given by Lemma~\ref{lem:oddcycles}. Let $\epsilon=\frac{1}{32k^2}$, and let the sets $A_1,\ldots,A_k$  of measure $\alpha>0$ be given by Lemma~\ref{lem:oddcycles}. 

In order to prove that $\frakc$ is not in the weak$^*$ closure of the convex hull of integral vertex covers, consider an arbitrary $\ell$-tuple of integral vertex covers $\frakc_1,\ldots,\frakc_\ell$ of $W$, and numbers $\gamma_1,\ldots,\gamma_\ell\ge 0$ with $\sum \gamma_i=1$.

Consider an arbitrary $i\in[\ell]$. We say that $\frakc_i$ \emph{marks the set $A_h$} (where $h\in[k]$) if $\frakc_i$ restricted to $A_h$ attains the value~0 on a set of measure at most $\frac{\alpha}{4k}$. Put equivalently, $\frakc_i$ marks the set $A_h$ if $\frakc_i$ restricted to $A_h$ attains the value~1 on a set of measure at least  $\frac{\alpha(4k-1)}{4k}$.
\begin{claim}\label{cl:mark}
For each $i\in[\ell]$, the vertex cover $\frakc_i$ marks at least $\frac{k+1}{2}$ many of the sets $A_1,\ldots,A_k$.
\end{claim}
\begin{proof}[Proof of Claim~\ref{cl:mark}]
Suppose that this is not the case. Recall that $k$ is odd. We can find an index $h\in[k]$ that $A_h$ and $A_{h+1}$ are not marked (again, using the cyclic notation $A_{k+1}=A_k$). Therefore, the $\frakc_i$-preimages  $B_h\subset A_h$ and $B_{h+1}\subset A_{h+1}$ of~0 have both measures more than $\frac{\alpha}{4k}$. It follows from Lemma~\ref{lem:oddcycles} and the way we set $\epsilon$ that $W$ is positive on a set of positive measure on $B_h\times B_{h+1}$. This contradicts the fact that $\frakc_i$ is a vertex cover.
\end{proof}
Let us write $A=\bigcup_h A_h$. It follows from Claim~\ref{cl:mark} that 
\begin{equation}\label{eq:vyuzijmarky}
\int_A \frakc_i \ge \frac{k+1}2\cdot \frac{\alpha(4k-1)}{4k}\ge \alpha\left(\frac{k}2+\frac{1}{6}\right)\;.
\end{equation}
By convexity, we can replace $\int_A\frakc_i$ by $\int_A\left(\sum_i\gamma_i\frakc_i\right)$ in~\eqref{eq:vyuzijmarky}.

We now have
\begin{equation}\label{eq:integralNaA}
\begin{split}
\int_A\left(\sum_i \gamma_i\frakc_i(x) - \frakc (x)\right)\:\mathrm{d}x=&
\int_A\left(\sum_i \gamma_i\frakc_i(x)\right)\:\mathrm{d}x - \int_A\frakc (x)\:\mathrm{d}x\\
\geByRef{eq:vyuzijmarky}&\alpha\left(\frac{k}2+\frac{1}{6}\right)-k\alpha\cdot\frac12=\frac{\alpha}{6}
\;.
\end{split}
\end{equation}
Since neither the set $A$ nor the bound on the right-hand side of~(\ref{eq:integralNaA}) depend on the choice of the number $\ell$, the vertex covers $\frakc_i$, and the constants $\gamma_i$, we get that $\frakc$ is not in the weak$^*$ closure of convex combinations of integral vertex covers, as was needed.
\end{proof}

\subsection{An application: the Erd\H os--Gallai Theorem}
In this section, we prove a graphon counterpart to the following classical result of Erd\H os and Gallai, \cite{Erdos1959}.
\begin{theorem}[Erd\H os--Gallai, 1959]\label{thm:ErdGall}
Suppose that $n$ and $\ell$ are positive integers that satisfy $\ell\le n/2$. Then
any $n$-vertex graph with more than $\max\{(\ell-1)(n-\ell+1)+{\ell-1\choose 2},{2\ell-1\choose 2}\}$ edges contains a matching with at least $\ell$ edges.
\end{theorem}
The bound in Theorem~\ref{thm:ErdGall} is optimal. Indeed, when $\ell\le 0.4(n+1)$, the \emph{extremal graph} (denoted by $\textrm{ExG}(n,\ell)$) is the complete graph $K_{\ell-1,n-\ell+1}$ together with the complete graph inserted into the $(\ell-1)$-part. When $\ell\ge 0.4(n+1)$, the extremal graph is the complete graph of order $2\ell-1$ with $n+1-2\ell$ isolated vertices padded. That is, the extremal graph undergoes a transition at edge density (asymptotically) 0.64.

Motivated by this, for $e\in[0,1]$ we define graphons $\Psi_{e}$ and $\Phi_{e}$ as follows. We partition $\Omega=B_1\sqcup B_2$ so that $\nu(B_1)=1-\sqrt{1-e}$ and $\nu(B_2)=\sqrt{1-e}$. We define $\Psi_{e}$ to be constant~0 on $B_2\times B_2$ and~1 elsewhere. We partition $\Omega=C_1\sqcup C_2$ so that $\nu(C_1)=e^2$ and $\nu(C_2)=1-e^2$. We define $\Phi_{e}$ to be constant~1 on $C_1\times C_1$ and~0 elsewhere. These definition uniquely determine $\Psi_{e}$ and $\Phi_{e}$, up to isomorphism. Thus, our graphon version of the Erd\H os--Gallai theorem reads as follows.
\begin{theorem}\label{thm:EGgraphon}
Suppose that $W:\Omega^2\rightarrow[0,1]$ is a graphon. Let $e= \int_{x,y} W(x,y)$. Then $\fraccoverratio(W)\ge \min\left\{\sqrt{\frac 14e},1-\sqrt{1-e}\right\}$. (The maximum is attained by the former term for $e\le 0.64$ and by the latter term for $e\ge 0.64$).

Furthermore, we have an equality if and only if $W$ is isomorphic to $\Psi_{e}$ (if $e\le 0.64$) or to $\Phi_{e}$ (if $e\ge 0.64$).
\end{theorem}
This version of the Erd\H os--Gallai Theorem implies an asymptotic version of the finite statement. Furthermore, it provides a corresponding stability statement.
\begin{theorem}\label{thm:EGEG}
For every $\epsilon>0$ there exists numbers $n_0\in\mathbb N$ and $\delta>0$ such that the following holds. If $G$ is a graph on $n>n_0$ vertices with more than
$\max\{(\ell-1)(n-\ell+1)+{\ell-1\choose 2},{2\ell-1\choose 2}\}$ edges, then $G$ contains a matching with at least $\ell-\epsilon n$ edges.
Furthermore, $G$ contains a matching with at least $\ell+\delta n$ edges, unless $G$ is $\epsilon n^2$-close to the graph $\mathrm{ExG}(n,\ell)$ as above in the edit distance.
\end{theorem}
The way of deriving Theorem~\ref{thm:EGEG} from Theorem~\ref{thm:EGgraphon} is standard, and we refer the reader to~\cite{HlHuPi:Komlos} where this was done in detail in the context of a tiling theorem of Koml\'os,~\cite{Komlos2000}, which is a statement of a similar flavor. 

Let us emphasize that the original proof of Theorem~\ref{thm:ErdGall} is simple and elementary (and the corresponding stability statement would not be difficult to prove with the same approach either). While our proof is not long, it makes use of the heavy machinery of graph limits, and in particular the results from~\cite{HlHuPi:TilingsInGraphons} and from Section~\ref{sub:graphonintegrality}. We think that our proof offers an interesting alternative point of view on the problem.

\medskip

For the proof of Theorem~\ref{thm:EGgraphon} we shall need the following fact.
\begin{fact}\label{fact:maxEG}
Suppose that $D\in[0,\frac12]$ is fixed. Then the maximum of the function $g(a,b)=a^2+2b-b^2$ on the set $\{(a,b): a\ge 0,b\ge 0,\frac{a}2+b=D\}$ is attained
for $a=0,b=D$ if $D\le 0.4$ and $a=2D,b=0$ if $D\ge 0.4$.
\end{fact}
\begin{proof}
We transform this into an optimization problem in one variable by considering the function $h(b)=g(2(D-b),b)$. The function $h$ is quadratic with limit plus infinity at $-\infty$ and at $+\infty$. Thus, the maximum of $h$ on the interval $[0,D]$ will be either at $b=0$ or at $b=D$. We have $h(0)=4D^2$ and $h(D)=2D-D^2$. A quick calculation gives that the latter is bigger for $D<0.4$ while the the latter is bigger for $D>0.4$.
\end{proof}
\begin{proof}[Proof of Theorem~\ref{thm:EGgraphon}]
By Theorem~\ref{thm:CovHalfIntegral} we can fix a half-integral vertex cover of size $\fraccoverratio(W)$. Then we get a partition $\Omega=A_0\sqcup A_{1/2} \sqcup A_1$ given by the preimages of 0, $1/2$ and 1. Let us write $\alpha_0=\nu(A_0)$, $\alpha_{1/2}=\nu(A_{1/2})$, and $\alpha_1=\nu(A_1)$. We have $\fraccoverratio(W)=\frac12\nu(A_{1/2})+\nu(A_{1})$. Note that $W_{\restriction A_0\times(A_0\cup A_{1/2})}=0$. Therefore,
\begin{align}
\begin{split}\label{eq:getback}
e\le \alpha_{1/2}^2+\alpha_1(\alpha_1+2\alpha_{1/2}+2\alpha_0)&=\alpha^2_{1/2}+2\alpha_1-\alpha_1^2\\
&\le \max\left\{4\fraccoverratio(W)^2,2\fraccoverratio(W)-\fraccoverratio(W)^2\right\}\;,
\end{split}
\end{align}
where the last part follows from Fact~\ref{fact:maxEG}. 
Pedestrian calculations show that this is equivalent to the assertion of the theorem.

We now look at the furthermore part. If $e= \max\{ 4\fraccoverratio(W)^2,2\fraccoverratio(W)-\fraccoverratio(W)^2\}$, then the first inequality in~\eqref{eq:getback} is at equality. This means that $W$ must be~1 on $A_{1/2}\times A_{1/2}$ and on $A_1\times \Omega$. Furthermore, Fact~\ref{fact:maxEG} tells us that for the second inequality in~\eqref{eq:getback} to be at equality, we must have $\alpha_{1/2}=0$ or $\alpha_1=1$. We conclude that $W$ is isomorphic to $\Psi_e$ or to $\Phi_e$.
\end{proof}

\section{Convergence of polytons}\label{sec:matchingpolytonconvergent}
\subsection{Fractional vertex cover polytons of a convergent graphon sequence}\label{sub:coverpolytonconvergent}
Suppose that a sequence of graphons $(W_n)_n$ converges to a graphon $W$ in the cut-norm. We want to relate the polytons $\fraccoverpolytope(W_n)$ to the polyton $\fraccoverpolytope(W)$. First observe that in general, the polytons $\fraccoverpolytope(W_n)$ do not converge to $\fraccoverpolytope(W)$ in any reasonable sense. Indeed, for example, take $W_n$ to be a representation of a sample of the Erd\H{o}s--R\'enyi random graph $\mathbb G(2n,1/\sqrt{n})$. It is well-known that almost surely almost all these graphs contain a perfect matching.\footnote{\label{f:almostallperfect}By saying that ``almost all these graphs contain a perfect matching'' we mean that with probability~1 (in the product probability space $\prod_n \mathbb G(2n,1/\sqrt{n})$), after sampling the graphs $G_{2n}\sim \mathbb G(2n,1/\sqrt{n})$, we can find a finite set $S\subset \mathbb{N}$ such that for each $n\in \mathbb N\setminus S$, the graph $G_{2n}$ contains a perfect matching. We provide a proof of this in Appendix~\ref{app:almostallperfect}.} Thus, $\fraccoverpolytope(W_n)$ contain only fractional vertex covers of size $\frac12$ and more for almost all $n$. On the other hand, almost surely, the zero graphon $W=0$ is the cut-distance limit of $(W_n)_n$, and so $\fraccoverpolytope(W)$ consists of all $[0,1]$-valued measurable functions on $\Omega$.

However, Theorem~\ref{thm:convergencepolytons} below shows that ``$\fraccoverpolytope(W)$ asymptotically contains the polytons $\fraccoverpolytope(W_n)$''. This theorem is a special case of~\cite[Theorem~3.14]{HlHuPi:TilingsInGraphons}.
\begin{theorem}\label{thm:convergencepolytons}
Suppose that $(W_n)_n$ is a sequence of graphons on $\Omega$ that converges to a graphon $W:\Omega^2\rightarrow[0,1]$ in the cut-norm. Suppose that $\frakc_n\in\fraccoverpolytope(W_n)$. Then any accumulation point of the sequence $(\frakc_n)_n$ in the weak$^*$ topology lies in $\fraccoverpolytope(W)$.
\end{theorem}

\subsection{Matching polytons of a convergent graphon sequence}
The main new result of this section concerns convergence properties of the matching polytons. This result is dual to Theorem~\ref{thm:convergencepolytons}: if $W_n$ converges to $W$ in the cut-norm then ``$\matchingpolytope(W_n)$  asymptotically contain $\matchingpolytope(W)$''.
\begin{theorem}\label{thm:convergencematchings}
Suppose that $W:\Omega^2\rightarrow[0,1]$ is a graphon on a probability space $\Omega$, and let $\frak m\in\matchingpolytope(W)$ be fixed. Then for every $\varepsilon>0$ there is $\delta>0$ such that whenever $U\colon\Omega^2\rightarrow [0,1]$ is a graphon with $\|U-W\|_{\square}<\delta$ then there is $\frak m_U\in\matchingpolytope(U)$ such that $\|\frak m_U-\frak m\|_{\square}<\varepsilon$.
\end{theorem}

Since the cut-norm topology is stronger than the weak$^*$ topology, we get the following corollary.
\begin{corollary}\label{cor:assymptcont}
Suppose that $(W_n)_n$ is a sequence of graphons on a probability space $\Omega$ that converges to a graphon $W:\Omega^2\rightarrow[0,1]$ in the cut-norm. Suppose that $\frak m\in\matchingpolytope(W)$. Then there exists a sequence $\frak m_n\in\matchingpolytope(W_n)$  such that $(\frak m_n)_n$ converges to $\frak m$ in the cut-norm. In particular, the sequence $(\frak m_n)_n$ converges to $\frak m$ in the weak$^*$ topology.
\end{corollary}

Recall that each function $F\in L^1(\Omega^2)$ can be approximated by a step-function up to an arbitrarily small error in the $L^1$-norm. Further, the approximating function can be chosen to equal the average on each step. In the proof of Theorem~\ref{thm:convergencematchings}, we will need the following extension of this fact: firstly, we want to approximate two functions $F_1,F_2\in L^1(\Omega^2)$ using the same steps, and secondly, we want the steps to be squares with sides of the same measure.
\begin{lemma}\label{lem:approximation_by_means}
Let $(\Omega,\nu)$ be a probability space. Then for every pair $(F_1,F_2)\in \LpSpace^1(\Omega^2)\times \LpSpace^1(\Omega^2)$ and every $\varepsilon>0$ there exists a number $k\in\mathbb{N}$ and a partition of $\Omega$ into $k$ pairwise disjoint subsets $\Omega_1,\ldots,\Omega_k$, each of measure 	$\frac 1k$, such that
\begin{equation}\label{eq:approximation_by_mean_values}
\sum\limits_{i,j=1}^k\int\limits_{\Omega_i\times\Omega_j}|F_1-F_1^{ij}|<\varepsilon\;\;\;\;\;\text{ and }\;\;\;\;\;\sum\limits_{i,j=1}^k\int\limits_{\Omega_i\times\Omega_j}|F_2-F_2^{ij}|<\varepsilon\;,
\end{equation}
where
\begin{equation}\nonumber
F_1^{ij}=\frac{1}{\nu(\Omega_i)\nu(\Omega_j)}\int\limits_{\Omega_i\times\Omega_j}F_1\;\;\;\;\;\text{ and }\;\;\;\;\;F_2^{ij}=\frac{1}{\nu(\Omega_i)\nu(\Omega_j)}\int\limits_{\Omega_i\times\Omega_j}F_2\;,\;\;\;\;\;i,j=1,\ldots,n\;.
\end{equation}
\end{lemma}
The proof of Lemma~\ref{lem:approximation_by_means} is given in Appendix.

\subsection{Proof of Theorem~\ref{thm:convergencematchings}}
Let $\varepsilon>0$ be fixed.
By basic properties of $\LpSpace^1$-functions there is $M>0$ (which we fix now) such that if we define
\begin{equation}\nonumber
\tilde{\frak m}(x,y)=\min\left\{\frak m(x,y),M\right\}
\end{equation}
then we have
\begin{equation}\label{eq:oriznutiShora}
\|\tilde{\frak m}-\frak m\|_{\LpSpace^1(\Omega^2)}<\frac 12\varepsilon\;.
\end{equation}
Moreover, it is obvious that such defined function $\tilde{\frak m}$ is still a matching in the graphon $W$.
We fix $\tilde\varepsilon>0$ such that
\begin{equation}\label{eq:choiceOfTildeEpsilon}
3\tilde\varepsilon+6\sqrt{\tilde\varepsilon}M+2\tilde\varepsilon^{\frac 32}<\frac 12\varepsilon\;.
\end{equation}

\begin{claim}\label{cl:oriznutiZdola}
	There is $r>0$ such that whenever $\Theta\subseteq\Omega^2$ is of  positive measure such that $\frac{1}{\nu^{\otimes 2}(\Theta)}\int\limits_{\Theta}W<r$ then
	\begin{equation}\nonumber
	\nu^{\otimes 2}\left(\SUPPORT(W)\cap\Theta\right)<\frac 1{2M}\tilde\varepsilon\;.
	\end{equation}
\end{claim}

\begin{proof}
	By basic properties of measurable functions, there is $s>0$ such that
	\begin{equation}\label{eq:maleHodnotyW}
	\nu^{\otimes 2}\{(x,y)\in\Omega^2\colon 0<W(x,y)<s\}<\frac 1{4M}\tilde\varepsilon\;.
	\end{equation}
	We will prove that $r=\frac 1{4M}\tilde\varepsilon s$ works. Suppose for a contradiction that there is $\Theta\subseteq\Omega^2$ of  positive measure with $\frac{1}{\nu^{\otimes 2}(\Theta)}\int\limits_{\Theta}W<r$ such that
	\begin{equation}\label{eq:predpokladProSpor}
	\nu^{\otimes 2}\left(\SUPPORT(W)\cap\Theta\right)\ge\frac 1{2M}\tilde\varepsilon\;.
	\end{equation}
	Then we have
	\begin{equation}\nonumber
	\begin{split}
	&r>\frac 1{\nu^{\otimes 2}(\Theta)}\int\limits_{\Theta}W\ge\frac 1{\nu^{\otimes 2}(\Theta)}\nu^{\otimes 2}\left(\{(x,y)\in\Theta\colon W(x,y)\ge s\}\right)\cdot s\\
	\stackrel{(\ref{eq:maleHodnotyW})}{>}&\frac 1{\nu^{\otimes 2}(\Theta)}\left(\nu^{\otimes 2}\left(\SUPPORT(W)\cap\Theta\right)-\frac 1{4M}\tilde\varepsilon\right)s\stackrel{(\ref{eq:predpokladProSpor})}{\ge}\frac 1{\nu^{\otimes 2}(\Theta)}\cdot\frac 1{4M}\tilde\varepsilon s\ge\frac 1{4M}\tilde\varepsilon s\;,
	\end{split}
	\end{equation}
	which is the desired contradiction with the definition of $r$.
\end{proof}

Now we fix $r>0$ from Claim~\ref{cl:oriznutiZdola}, and we set
\begin{equation}\label{eq:choiceOfEta}
\eta=\frac 12\tilde\varepsilon\cdot\frac 1{1+2M+2\frac Mr}\;.
\end{equation}
By Lemma~\ref{lem:approximation_by_means} there is a natural number $k$ and a partition of $\Omega$ into pairwise disjoint subsets $\Omega_1,\ldots,\Omega_k$, each of measure $\frac 1k$, such that
\begin{equation}\label{eq:approx_of_W_and_m}
\sum\limits_{i,j=1}^k\int\limits_{\Omega_i\times\Omega_j}|W-W^{ij}|<\eta^2\;\;\;\;\;\text{ and }\;\;\;\;\;\sum\limits_{i,j=1}^k\int\limits_{\Omega_i\times\Omega_j}|\tilde{\frak m}-m^{ij}|<\eta^2\;,
\end{equation}
where
\begin{equation}\nonumber
W^{ij}=\frac{1}{\nu(\Omega_i)\nu(\Omega_j)}\int\limits_{\Omega_i\times\Omega_j}W\;\;\;\;\;\text{ and }\;\;\;\;\;m^{ij}=\frac{1}{\nu(\Omega_i)\nu(\Omega_j)}\int\limits_{\Omega_i\times\Omega_j}\tilde{\frak m}\;,\;\;\;\;\;i,j=1\ldots,n\;.
\end{equation}
The first inequality from (\ref{eq:approx_of_W_and_m}) easily implies that for all but at most $\eta k^2$ pairs $(i,j)$ we have
\begin{equation}\label{eq:app_W^ij}
\int\limits_{\Omega_i\times\Omega_j}\left|W-W^{ij}\right|\le\frac 1{k^2}\eta\;.
\end{equation}
Similarly, the second inequality from (\ref{eq:approx_of_W_and_m}) implies that for all but at most $\eta k^2$ pairs $(i,j)$ we have
\begin{equation}\label{eq:app_m^ij}
\int\limits_{\Omega_i\times\Omega_j}\left|\frak m-m^{ij}\right|\le\frac 1{k^2}\eta\;.
\end{equation}

We are now in a position, when we can find $\delta>0$ with the properties required by Theorem~\ref{thm:convergencematchings}. Set
\begin{equation}\label{eq:choiceOfDelta}
\delta=\frac 1{k^2}\eta\;.
\end{equation}
So, let $U\colon\Omega^2\rightarrow[0,1]$ be a graphon such that $\|U-W\|_{\square}<\delta$.
Let $A\subset [k]\times[k]$ denote the set of all those pairs $(i,j)$ for which either (\ref{eq:app_W^ij}) or (\ref{eq:app_m^ij}) fails. We have that $|A|\le 2\eta k^2$.
We define
\begin{equation}\nonumber
t(x,y)=
\begin{cases}
\dfrac{m^{ij}}{W^{ij}}\cdot U(x,y)&\text{ if }(i,j)\notin A\text{ and }W^{ij}\ge r,\\
0&\text{ otherwise},
\end{cases}
\;\;\;\;\;(x,y)\in\Omega_i\times\Omega_j,\;\;\;\;\;i,j=1,\ldots,k\;.
\end{equation}

\begin{claim}\label{cl:tJeBlizko}
	We have that $\|t-\tilde{\frak m}\|_{\square}\le\tilde\varepsilon$.
\end{claim}

\begin{proof}
	We need to show that for every measurable sets $S,T\subseteq\Omega$ it holds $\left|\int\limits_{S\times T}(t-\tilde{\frak m})\right|\le\tilde\varepsilon$. So let us fix the sets $S,T\subseteq\Omega$.
	Let $\Theta_A$ denote the union of all the sets $\Omega_i\times\Omega_j$ for which $(i,j)\in A$.
	Similarly, let $\Theta_B$ denote the union of all the sets $\Omega_i\times\Omega_j$ for which $(i,j)\notin A$ and $W^{ij}<r$, and let $\Theta_C$ denote the union of all the sets $\Omega_i\times\Omega_j$ for which $(i,j)\notin A$ and $W^{ij}\ge r$.
	
	The bulk of the work is in proving the following three subclaims.
	\begin{subclaim}\label{claim:martin1}
		We have
		$$\left|\int\limits_{(S\times T)\cap\Theta_A}(t-\tilde{\frak m})\right|\le2\eta M\;.$$
	\end{subclaim}
	\begin{subclaim}\label{claim:martin2}
		We have
		$$
		\left|\int\limits_{(S\times T)\cap\Theta_B}(t-\tilde{\frak m})\right|<\frac 12\tilde\varepsilon\;.
		$$
	\end{subclaim}
	\begin{subclaim}\label{claim:martin3}
		We have
		$$
		\left|\int\limits_{(S\times T)\cap\Theta_C}(t-\tilde{\frak m})\right|<\eta\left(1+2\frac Mr\right)\;.
		$$
	\end{subclaim}	
	Indeed, Subclaims~\ref{claim:martin1}--\ref{claim:martin3} complete the proof of Claim~\ref{cl:tJeBlizko} as then
	\begin{equation}\nonumber
	\left|\int\limits_{S\times T}(t-\tilde{\frak m})\right|<\eta\left(1+2M+2\frac Mr\right)+\frac 12\tilde\varepsilon\stackrel{(\ref{eq:choiceOfEta})}{\le}\tilde\varepsilon\;.
	\end{equation}
	
\begin{claimproof}[Proof of Subclaim~\ref{claim:martin1}]
Recall that by the definition of the set $A$, it holds $\nu^{\otimes 2}(\Theta_A)\le 2\eta$, and so we have
	\begin{equation}\nonumber
	\left|\int\limits_{(S\times T)\cap\Theta_A}(t-\tilde{\frak m})\right|=\int\limits_{(S\times T)\cap\Theta_A}\tilde{\frak m}\le\int\limits_{\Theta_A}\tilde{\frak m}\le\nu^{\otimes 2}(\Theta_A)M\le2\eta M\;.
	\end{equation}
\end{claimproof}
	
\begin{claimproof}[Proof of Subclaim~\ref{claim:martin2}]
 If $\nu^{\otimes 2}(\Theta_B)=0$ then trivially
	\begin{equation}\nonumber
	\left|\int\limits_{(S\times T)\cap\Theta_B}(t-\tilde{\frak m})\right|=0\;.
	\end{equation}
	So suppose that $\Theta_B$ is of positive measure. Note that then it clearly holds $\frac 1{\nu^{\otimes 2}(\Theta_B)}\int\limits_{\Theta_B}W<r$, and so we have
	\begin{equation}\nonumber
	\begin{split}
	\left|\int\limits_{(S\times T)\cap\Theta_B}(t-\tilde{\frak m})\right|&=\int\limits_{(S\times T)\cap\Theta_B}\tilde{\frak m}\le\nu^{\otimes 2}(\SUPPORT(\tilde{\frak m})\cap\Theta_B)\cdot M\\
	&\le\nu^{\otimes 2}(\SUPPORT(W)\cap\Theta_B)\cdot M\overset{\text{Claim~\ref{cl:oriznutiZdola}}}{<}\frac 12\tilde\varepsilon\;.
	\end{split}
	\end{equation}
\end{claimproof}
	
\begin{claimproof}[Proof of Sublclaim~\ref{claim:martin3}]
It is enough to show that whenever a pair $(i,j)\notin A$ is such that $W^{ij}\ge r$ then
	\begin{equation}\nonumber
	\left|\int\limits_{(S\times T)\cap(\Omega_i\times\Omega_j)}(t-\tilde{\frak m})\right|<\frac 1{k^2}\eta\left(1+2\frac Mr\right)\;.
	\end{equation}
	So let us fix such a pair $(i,j)$. Then we have
	\begin{equation}\nonumber
	\begin{split}
	&\left|\int\limits_{(S\times T)\cap(\Omega_i\times\Omega_j)}(t-\tilde{\frak m})\right|\\
	\le&\left|\int\limits_{(S\times T)\cap(\Omega_i\times\Omega_j)}\left(\frac{m^{ij}}{W^{ij}}U-m^{ij}\right)\right|+\left|\int\limits_{(S\times T)\cap(\Omega_i\times\Omega_j)}\left(m^{ij}-\tilde{\frak m}\right)\right|\\
	\le&\frac {m^{ij}}{W^{ij}}\left|\int\limits_{(S\times T)\cap(\Omega_i\times\Omega_j)}\left(U-W^{ij}\right)\right|+\int\limits_{(S\times T)\cap(\Omega_i\times\Omega_j)}|m^{ij}-\tilde{\frak m}|\\
	\le&\frac Mr\left|\int\limits_{(S\times T)\cap(\Omega_i\times\Omega_j)}\left((U-W)+(W-W^{ij})\right)\right|+\int\limits_{(S\times T)\cap(\Omega_i\times\Omega_j)}|m^{ij}-\tilde{\frak m}|\\
\le&\frac Mr\left|\int\limits_{(S\times T)\cap(\Omega_i\times\Omega_j)}(U-W)\right|+\frac Mr\int\limits_{(S\times T)\cap(\Omega_i\times\Omega_j)}|W-W^{ij}|+\int\limits_{(S\times T)\cap(\Omega_i\times\Omega_j)}|m^{ij}-\tilde{\frak m}|\\
\JUSTIFY{$(i,j)\notin A$}\le&\frac Mr\left|\int\limits_{(S\cap\Omega_i)\times (T\cap\Omega_j)}(U-W)\right|+\frac Mr\cdot\frac 1{k^2}\eta+\frac 1{k^2}\eta\\
\JUSTIFY{by (\ref{eq:choiceOfDelta})}<&\frac Mr\cdot\frac 1{k^2}\eta+\frac Mr\cdot\frac 1{k^2}\eta+\frac 1{k^2}\eta\\
	=&\frac 1{k^2}\eta\left(1+2\frac Mr\right)\;.
	\end{split}
	\end{equation}
\end{claimproof}
\end{proof}

Let $B_1$ denote the set of all those $x\in\Omega$ for which $$\int\limits_{y\in\Omega}t(x,y)>\int\limits_{y\in\Omega}\tilde{\frak m}(x,y)+\sqrt{\tilde\varepsilon}\;.$$ Similarly, let $B_2$ denote the set of all those $x\in\Omega$ for which $$\int\limits_{y\in\Omega}t(y,x)>\int\limits_{y\in\Omega}\tilde{\frak m}(y,x)+\sqrt{\tilde\varepsilon}\;.$$
Then we have
\begin{equation}\nonumber
\tilde\varepsilon\stackrel{\text{Claim~\ref{cl:tJeBlizko}}}{\ge}\|t-\tilde{\frak m}\|_{\square}\ge\int\limits_{B_1\times\Omega}(t-\tilde{\frak m})>\nu(B_1)\sqrt{\tilde\varepsilon}\;,
\end{equation}
and consequently $\nu(B_1)<\sqrt{\tilde\varepsilon}$. In the same way, we conclude that $\nu(B_2)<\sqrt{\tilde\varepsilon}$.
Now we are ready to define $\frak m_U$ by setting
\begin{equation}\label{eq:defOfm_U}
\frak m_U(x,y)=
\begin{cases}
0&\text{if }x\in B_1\cup B_2\text{ or }y\in B_1\cup B_2,\\
\frac 1{1+2\sqrt{\tilde\varepsilon}}t(x,y)&\text{otherwise},
\end{cases}
\;\;\;\;\;(x,y)\in\Omega^2\;.
\end{equation}

\begin{claim}\label{cl:m_U-t}
We have that $\|\frak m_U-t\|_{\LpSpace^1(\Omega^2)}<6\sqrt{\tilde\varepsilon}M+2\tilde\varepsilon+2\tilde\varepsilon^{\frac 32}$.
\end{claim}

\begin{proof}
	We set $B=\left((B_1\cup B_2)\times\Omega\right)\cup\left(\Omega\times(B_1\cup B_2)\right)$.
	Then we have
	\begin{align*}
	&\|\frak m_U-t\|_{\LpSpace^1(\Omega^2)}=\int\limits_{B}|\frak m_U-t|+\int\limits_{\Omega^2\setminus B}|\frak m_U-t|\\
	=&
	\int\limits_Bt+\left(1-\frac 1{1+2\sqrt{\tilde\varepsilon}}\right)\int\limits_{\Omega^2\setminus B}t\\
	=&\int\limits_{(B_1\cup B_2)\times\Omega}t+\int\limits_{\Omega\times(B_1\cup B_2)}t+\left(1-\frac 1{1+2\sqrt{\tilde\varepsilon}}\right)\int\limits_{(\Omega\setminus(B_1\cup B_2))^2}t\\
	\JUSTIFY{Claim~\ref{cl:tJeBlizko}}\le&\int\limits_{(B_1\cup B_2)\times\Omega}\tilde{\frak m}+\tilde\varepsilon+\int\limits_{\Omega\times(B_1\cup B_2)}\tilde{\frak m}+\tilde\varepsilon+\left(1-\frac 1{1+2\sqrt{\tilde\varepsilon}}\right)\left(\int\limits_{(\Omega\setminus(B_1\cup B_2))^2}\tilde{\frak m}+\tilde\varepsilon\right)\\
	\le&\nu^{\otimes 2}((B_1\cup B_2)\times\Omega)\cdot M+\nu^{\otimes 2}(\Omega\times(B_1\cup B_2))\cdot M+2\tilde\varepsilon+\frac{2\sqrt{\tilde\varepsilon}}{1+2\sqrt{\tilde\varepsilon}}(M+\tilde\varepsilon)\\
	<&4\sqrt{\tilde\varepsilon}M+2\tilde\varepsilon+2\sqrt{\tilde\varepsilon}(M+\tilde\varepsilon)=6\sqrt{\tilde\varepsilon}M+2\tilde\varepsilon+2\tilde\varepsilon^{\frac 32}\;.
\end{align*}
\end{proof}

Now we have
\begin{align*}
\|\frak m_U-\frak m\|_{\square}&\le\|\frak m_U-t\|_{\square}+\|t-\tilde{\frak m}\|_{\square}+\|\tilde{\frak m}-\frak m\|_{\square}\\
&\le\|\frak m_U-t\|_{\LpSpace^1(\Omega^2)}+\|t-\tilde{\frak m}\|_{\square}+\|\tilde{\frak m}-\frak m\|_{\LpSpace^1(\Omega^2)}\\
\JUSTIFY{by (\ref{eq:oriznutiShora}),\text{ Claims }\ref{cl:tJeBlizko}\text{ and }\ref{cl:m_U-t}}
&<\frac 12\varepsilon+3\tilde\varepsilon+6\sqrt{\tilde\varepsilon}M+2\tilde\varepsilon^{\frac 32}\\
\JUSTIFY{(\ref{eq:choiceOfTildeEpsilon})}&<\varepsilon\;.
\end{align*}
So it remains to show that $\frak m_U$ is a matching in the graphon $U$.

The fact that $\frak m_U$ is a nonnegative function from $\LpSpace^1(\Omega^2)$ is obvious, and we also have $\SUPPORT(\frak m_U)\subseteq\SUPPORT(t)\subseteq\SUPPORT(U)$. So we only need to show that for almost every $x\in\Omega$ it holds
\begin{equation}\label{eq:jeToMatching}
\int\limits_{y\in\Omega}\frak m_U(x,y)+\int\limits_{y\in\Omega}\frak m_U(y,x)\le 1\;.
\end{equation}
This is trivially satisfied for every $x\in B_1\cup B_2$ as then the left-hand side of (\ref{eq:jeToMatching}) equals $0$. So let us fix $x\in\Omega\setminus(B_1\cup B_2)$. We may assume that
\begin{equation}\label{eq:tildeMJeMatching}
\int\limits_{y\in\Omega}\tilde{\frak m}(x,y)+\int\limits_{y\in\Omega}\tilde{\frak m}(y,x)\le 1\;,
\end{equation}
as $\tilde{\frak m}$ is a matching (in the graphon $W$).
Then it holds
\begin{align*}
\int\limits_{y\in\Omega}\frak m_U(x,y)+\int\limits_{y\in\Omega}\frak m_U(y,x)&\stackrel{(\ref{eq:defOfm_U})}{\le}\frac 1{1+2\sqrt{\tilde\varepsilon}}\left(\int\limits_{y\in\Omega}t(x,y)+\int\limits_{y\in\Omega}t(y,x)\right)\\
\JUSTIFY{$x\notin B_1\cup B_2$}&\le\frac 1{1+2\sqrt{\tilde\varepsilon}}\left(\int\limits_{y\in\Omega}\tilde{\frak m}(x,y)+\int\limits_{y\in\Omega}\tilde{\frak m}(y,x)+2\sqrt{\tilde\varepsilon}\right)\stackrel{(\ref{eq:tildeMJeMatching})}{\le}1\;,
\end{align*}
which completes the proof of Theorem~\ref{thm:convergencematchings}.

\section{Concluding remarks}\label{sec:concluding}
\subsection{Approximating $\matchingpolytope(W)$ and $\fraccoverpolytope(W)$ using $W$-random graphs}
In Section~\ref{sec:matchingpolytonconvergent}, we showed that if $(W_n)_n$ converge to $W$ in the cut-norm then, in a certain sense, $\matchingpolytope(W_n)$ asymptotically contain $\matchingpolytope(W)$, and $\fraccoverpolytope(W_n)$ are asymptotically contained in $\fraccoverpolytope(W)$. We also showed that in general, these inclusions may be proper. However, we believe that if we take $W_n$ as a representation of $\mathbb{G}(n,W)$ then with probability one both these inclusions are asymptotically at equality. 

\subsection{Bipartiteness from the matching polyton}
Theorems~\ref{thm:CovIntegral} and~\ref{thm:CovIntegralImpliesBipartite} characterize bipartiteness of a graphon in terms of its fractional vertex cover polyton. For finite graphs there is another characterization in terms of the matching polytope: a graph is bipartite if and only if $\matchingpolytope(G)$ is integral. Recall that there seems to be no counterpart to the concept of integrality of a graphon matching (c.f. Remark~\ref{rem:intVSfrac}). So, we leave as an important question to provide a characterization of bipartiteness in terms of $\matchingpolytope(W)$.
\subsection{Perfect matching polyton}
Many variants of the above polytopes are considered in combinatorial optimization. As an example, let us mention the \emph{perfect matching polytope} $\perfmatchingpolytope(G)$ and the \emph{fractional perfect matching polytope} $\fracperfmatchingpolytope(G)$ of a graph $G$. The corresponding graphon polyton is
$$\perfmatchingpolytope(W)=\left\{\mathfrak{m}\in\matchingpolytope(W):\|\mathfrak m\|=1\right\}\;.
$$
(Again, we cannot distinguish between the integral and fractional version, cf. the discussion in Section~\ref{sec:matchings}.) It might be interesting to study this, and similar polytons. That said, let us emphasize that many basic results, like Edmonds' perfect matching polytope theorem, seem not to have a graphon counterpart as they concern integrality-related properties of the polytope.
\subsection{Generalizing the results to $F$-tilings}\label{sub:GenerelazingToFTilings}
Results in Section~\ref{sec:extremepoints} are specific to matchings --- even in the finite setting. Even though we have not worked out the details, we believe that our second main result, Theorem~\ref{thm:convergencematchings}, extends to general $F$-tilings as introduced in~\cite{HlHuPi:TilingsInGraphons} (and so does its proof).
\subsection{Extreme points of fractional vertex cover polytons}\label{exposed}
As we explained in Section~\ref{sub:graphonintegrality}, every extreme point of the fractional vertex cover polyton of a bipartite graphon is exposed.
We leave it as an open question whether every extreme point of the fractional vertex cover polyton is exposed even for non-bipartite graphons.
\section*{Acknowledgment}
We thank an anonymous referee who provided very detailed and helpful comments.
\bigskip

The contents of this publication reflects only the authors' views and not necessarily the views of the European Commission of the European Union.


\bibliographystyle{amsplain} \bibliography{../bibl}

\providecommand{\bysame}{\leavevmode\hbox to3em{\hrulefill}\thinspace}
\providecommand{\MR}{\relax\ifhmode\unskip\space\fi MR }
\providecommand{\MRhref}[2]{%
  \href{http://www.ams.org/mathscinet-getitem?mr=#1}{#2}
}
\providecommand{\href}[2]{#2}
\begin{thebibliography}{10}

\bibitem{Borgs2008c}
C.~Borgs, J.~T. Chayes, L.~Lov{\'a}sz, V.~T. S{\'o}s, and K.~Vesztergombi,
  \emph{{Convergent sequences of dense graphs. {I}. {S}ubgraph frequencies,
  metric properties and testing}}, Adv. Math. \textbf{219} (2008), no.~6,
  1801--1851.

\bibitem{ChatVar:LargeDev}
S.~Chatterjee and S.~R.~S. Varadhan, \emph{{The large deviation principle for
  the {E}rd\H{o}s-{R}{\'e}nyi random graph}}, European J. Combin. \textbf{32}
  (2011), no.~7, 1000--1017.

\bibitem{DHM:CliquesRandom}
M.~Dole\v{z}al, J.~Hladk{\'y}, and A.~M{\'a}th{\'e}, \emph{Cliques in dense
  inhomogeneous random graphs}, Random Structures Algorithms \textbf{51}
  (2017), no.~2, 275--314.

\bibitem{DoHlHuPi:CombinatorialOptimization}
M.~Dole{\v z}al, J.~Hladk{\'y}, P.~Hu, and D.~Piguet, \emph{First steps in
  combinatorial optimization on graphons: Matchings}, European {C}onference on
  {C}ombinatorics, {G}raph {T}heory and {A}pplications ({E}uro{C}omb 2017),
  Electron. Notes Discrete Math., vol.~61, Elsevier Sci. B. V., Amsterdam,
  2017, pp.~359--365.

\bibitem{Erdos1959}
P.~Erd\H{o}s and T.~Gallai, \emph{{On maximal paths and circuits of graphs}},
  Acta Math. Acad. Sci. Hungar \textbf{10} (1959), 337--356 (unbound insert).

\bibitem{HlHuPi:Komlos}
J.~Hladk{\'y}, P.~Hu, and D.~Piguet, \emph{{Koml{\'o}s's tiling theorem via
  graphon covers}}, to appear in J. Graph Theory, DOI: 10.1002/jgt.22365.

\bibitem{HlHuPi:TilingsInGraphons}
\bysame, \emph{{Tilings in graphons}}, arXiv:1606.03113.

\bibitem{HlaRoch:IndepCliCol}
J.~Hladk{\'y} and I.~Rocha, \emph{Independent sets, cliques, and colorings in
  graphons}, arXiv:1712.07367.

\bibitem{Komlos2000}
J.~Koml{\'o}s, \emph{{Tiling {T}ur{\'a}n theorems}}, Combinatorica \textbf{20}
  (2000), no.~2, 203--218.

\bibitem{Komlos1997}
J.~Koml{\'o}s, G.~N. S{\'a}rk{\"o}zy, and E.~Szemer{\'e}di, \emph{{Blow-up
  lemma}}, Combinatorica \textbf{17} (1997), no.~1, 109--123.

\bibitem{Lov:Sidorenko}
L.~Lov{\'a}sz, \emph{{Subgraph densities in signed graphons and the local
  {S}imonovits-{S}idorenko conjecture}}, Electron. J. Combin. \textbf{18}
  (2011), no.~1, Paper 127, 21.

\bibitem{Lovasz2012}
\bysame, \emph{{Large networks and graph limits}}, {American Mathematical
  Society Colloquium Publications}, vol.~60, American Mathematical Society,
  Providence, RI, 2012.

\bibitem{Lovasz2006}
L.~Lov{\'a}sz and B.~Szegedy, \emph{{Limits of dense graph sequences}}, J.
  Combin. Theory Ser. B \textbf{96} (2006), no.~6, 933--957.

\bibitem{Lovasz2010}
\bysame, \emph{{Regularity partitions and the topology of graphons}}, {An
  irregular mind}, {Bolyai Soc. Math. Stud.}, vol.~21, J{\'a}nos Bolyai Math.
  Soc., Budapest, 2010, pp.~415--446.

\bibitem{Posa1976}
L.~P{\'o}sa, \emph{{Hamiltonian circuits in random graphs}}, Discrete Math.
  \textbf{14} (1976), no.~4, 359--364.

\bibitem{Razborov2007}
A.~A. Razborov, \emph{{Flag algebras}}, J. Symbolic Logic \textbf{72} (2007),
  no.~4, 1239--1282.

\bibitem{Schrijver2003}
A.~Schrijver, \emph{{Combinatorial optimization. {P}olyhedra and efficiency.
  {V}ol. {A}}}, {Algorithms and Combinatorics}, vol.~24, Springer-Verlag,
  Berlin, 2003, Paths, flows, matchings, Chapters 1--38.

\end{thebibliography}

\appendix
\section{Proof of Lemma~\ref{lem:approximation_by_means}}
Let $\mathcal L$ be the family of all pairs $(F_1,F_2)\in \LpSpace^1(\Omega^2)\times \LpSpace^1(\Omega^2)$ such that for every $\varepsilon>0$ there exist a partition of $\Omega$ into finitely many pairwise disjoint subsets $B_1,\ldots,B_m$ (for a suitable natural number $m$), and real numbers $D_1^{pq}$, $D_2^{pq}$, $p,q=1,\ldots,m$, such that
\begin{equation}\nonumber
\sum\limits_{p,q=1}^m\int\limits_{B_p\times B_q}|F_1-D_1^{pq}|<\varepsilon\;\;\;\;\;\text{ and }\;\;\;\;\;\sum\limits_{p,q=1}^m\int\limits_{B_p\times B_q}|F_2-D_2^{pq}|<\varepsilon\;.
\end{equation}
It is easy to verify that $\mathcal L$ is a closed subspace of $\LpSpace^1(\Omega^2)\times \LpSpace^1(\Omega^2)$ containing all pairs consisting of characteristic functions of measurable rectangles. Therefore it holds $\mathcal L=\LpSpace^1(\Omega^2)\times \LpSpace^1(\Omega^2)$.

Now let us fix $(F_1,F_2)\in \LpSpace^1(\Omega^2)\times \LpSpace^1(\Omega^2)$ and $\varepsilon>0$. As $(F_1,F_2)\in\mathcal L$, we can find a partition of $\Omega$ into finitely many pairwise disjoint subsets $B_1,\ldots,B_m$, and real numbers $D_1^{pq}$, $D_2^{pq}$, $p,q=1,\ldots,m$, such that
\begin{equation}\label{eq:app1}
\sum\limits_{p,q=1}^m\int\limits_{B_p\times B_q}|F_1-D_1^{pq}|<\frac 14\varepsilon\;\;\;\;\;\text{ and }\;\;\;\;\;\sum\limits_{p,q=1}^m\int\limits_{B_p\times B_q}|F_1-D_2^{pq}|<\frac 14\varepsilon\;.
\end{equation}
By the absolute continuity of the Lebesgue integral there is $\delta>0$ such that $\int\limits_E|F_1|<\frac 14\varepsilon$ and $\int\limits_E|F_2|<\frac 14\varepsilon$, whenever $E\subseteq\Omega\times\Omega$ is such that $\nu^{\otimes 2}(E)<\delta$.
We fix a natural number $k$ such that $\frac mk<\frac{1}2\delta$. For every $p=1,\ldots,m$, we find a decomposition $B_p=B_p^0\sqcup B_p^1\sqcup\ldots\sqcup B_p^{n_p}$ of $B_p$ into finitely many pairwise disjoint subsets such that $\nu(B_p^r)=\frac 1k$ for $r=1,\ldots,n_p$, and $\nu(B_p^0)\le\frac 1k$. We set $F=\bigcup\limits_{p=1}^mB_p^0$. Then $\nu(F)$ is smaller that $\frac{1}2\delta$, and so
\begin{equation}\label{eq:abs_cont}
\int\limits_{(F\times\Omega)\cup(\Omega\times F)}|F_1|<\frac 14\varepsilon\;\;\;\;\;\text{ and }\;\;\;\;\;\int\limits_{(F\times\Omega)\cup(\Omega\times F)}|F_2|<\frac 14\varepsilon\;.
\end{equation}
Moreover, $\nu(F)$
is a multiple of $\frac{1}{k}$, and so it can be decomposed into finitely many disjoint subsets $\tilde B_1,\ldots,\tilde B_l$ (for a suitable natural number $l$), each of measure $\frac 1k$.
Let $\Omega_1,\ldots\Omega_k$ be some enumeration of the sets $\tilde B_1\ldots,\tilde B_l$ and $B_p^r$, $r=1\ldots,n_p$, $p=1,\ldots,m$. For every $i,j=1,\ldots,k$, we define real numbers $C_1^{ij}$ and $C_2^{ij}$ as follows. If $\Omega_i\times\Omega_j=B_p^r\times B_q^s$ for some $p,q\in\{1,\ldots,m\}$, $r\in\{1,\ldots,n_p\}$ and $s\in\{1,\ldots,n_q\}$ then we set $C_1^{ij}=D_1^{pq}$ and $C_2^{ij}=D_2^{pq}$. Otherwise, we set $C_1^{ij}=C_2^{ij}=0$. Let us fix $t\in\{0,1\}$.
Then we have
\begin{equation}\label{eq:aproximation_by_step_function}
\begin{split}
&\sum\limits_{i,j=1}^k\int\limits_{\Omega_i\times\Omega_j}|F_t-C_t^{ij}|\\
=&\sum\limits_{p,q=1}^m\sum\limits_{\substack{r=1,\ldots,n_p\\s=1,\ldots,n_q}}\int\limits_{B_p^r\times B_q^s}|F_t-D_t^{pq}|+\int\limits_{(F\times\Omega)\cup(\Omega\times F)}|F_t|\\
\stackrel{(\ref{eq:app1}),(\ref{eq:abs_cont})}{<}&\frac 14\varepsilon+\frac 14\varepsilon=\frac 12\varepsilon\;.
\end{split}
\end{equation}

\begin{claim}\label{cl:mean_value}
For every $i,j=1,\ldots,k$, we have that
\begin{equation}\nonumber
\int\limits_{\Omega_i\times\Omega_j}|F_t-F_t^{ij}|\le2\int\limits_{\Omega_i\times\Omega_j}|F_t-C_t^{ij}|\;.
\end{equation}
\end{claim}
\begin{proof}
Let us fix $i,j\in\{1,\ldots,k\}$. It holds
\begin{equation}\label{eq:linearity_of_integral}
\int\limits_{\Omega_i\times\Omega_j}|F_t-F_t^{ij}|=\int\limits_{\substack{(x,y)\in\Omega_i\times\Omega_j\\F_t(x,y)<F_t^{ij}}}(F_t^{ij}-F_t)+\int\limits_{\substack{(x,y)\in\Omega_i\times\Omega_j\\F_t(x,y)>F_t^{ij}}}(F_t-F_t^{ij})\;,
\end{equation}
and the two integrals on the right hand side of (\ref{eq:linearity_of_integral}) equals each other (by the definition of $F_t^{ij}$). Therefore it is enough to show that one of these integrals is less or equal to $\int\limits_{\Omega_i\times\Omega_j}|F_t-C_t^{ij}|$. Assume for example that $C_t^{ij}\ge F_t^{ij}$ (the complementary case is similar). Then we have
\begin{equation}\nonumber
\int\limits_{\substack{(x,y)\in\Omega_i\times\Omega_j\\F_t(x,y)<F_t^{ij}}}(F_t^{ij}-F_t)\le\int\limits_{\substack{(x,y)\in\Omega_i\times\Omega_j\\F_t(x,y)<C_t^{ij}}}(C_t^{ij}-F_t)\le\int\limits_{\Omega_i\times\Omega_j}|F_t-C_t^{ij}|\;,
\end{equation}
as we wanted.
\end{proof}

Inequality~(\ref{eq:aproximation_by_step_function}) combined with Claim~\ref{cl:mean_value} gives us (\ref{eq:approximation_by_mean_values}).

\section{Details concerning Footnote~\ref{f:almostallperfect}}\label{app:almostallperfect}
Recall a celebrated theorem of P\'osa~\cite{Posa1976} that there exists a constant $C>0$ so that the random graph $\mathbb{G}(2n,\frac{C\log n}n)$ contains a Hamilton cycle with probability $1-o_n(1)>\frac12$. Note that Hamiltonicity implies the existence of a perfect matching. Let $E_{2n}$ be the union of edges obtained in $\sqrt[3]{n}$ independent copies of $\mathbb{G}(2n,C\log n/n)$. Since we $\sqrt[3]{n}\cdot \frac{C\log n}n<1/\sqrt{n}$, get that the edge set $E_{2n}$ is stochastically dominated (with respect to inclusion) by $\mathbb{G}(2n,1/\sqrt{n})$. Hence the probability that $\mathbb{G}(2n,1/\sqrt{n})$ does not contain a perfect matching is at most $$p_{2n}:=\left(\frac{1}{2}\right)^{\sqrt[3]{n}}\;.$$
The sequence $(p_{2n})_n$ is summable, and hence the Borel--Cantelli lemma tells us that at most finitely many ``bad events'' occur.
\end{document}